\documentclass{siamart1116}
\RequirePackage[OT1]{fontenc}
\usepackage{graphicx}
\usepackage{latexsym,color}
\usepackage{amssymb}
\usepackage{amsfonts}
\usepackage{amsmath}
\usepackage{float}
\newtheorem{thm}{Theorem}[section]
 
 \newtheorem{lem}{Lemma}[section]
 \newtheorem{rem}{Remark}[section]

\newtheorem{asm}{Assumption}[section]

\numberwithin{equation}{section}

\renewcommand {\thetable} {\thesection{}.\arabic{table}}
\renewcommand {\thefigure} {\thesection{}.\arabic{figure}}
\renewcommand{\theequation}{\arabic{section}.\arabic{equation}}

\begin{document}

\title{Recombining Tree Approximations for Optimal Stopping for Diffusions}
\author{Erhan Bayraktar\footnotemark[2]\ \footnotemark[5]
\and Yan Dolinsky\footnotemark[3]\ \footnotemark[5]
\and Jia Guo\footnotemark[4]}
\date{\today}

\maketitle \markboth{ Bayraktar,Dolinsky, and Guo }{Recombining Tree Approximations}
\renewcommand{\thefootnote}{\fnsymbol{footnote}}
\footnotetext[5]{E. Bayraktar is supported in part by the National Science Foundation under grant DMS-1613170 and by the Susan M. Smith Professorship.
Y. Dolinsky is supported by supported by Marie–Curie Career Integration Grant, no. 618235 and the ISF grant no 160/17}
\footnotetext[2]{Department of Mathematics, University of Michigan, \email{erhan@umich.edu}.}
\footnotetext[3]{Department of Statistics, Hebrew University and School of Mathematical Sciences, Monash University \email{yan.dolinsky@mail.huji.ac.il}.}
\footnotetext[4]{Department of Mathematics, University of Michigan, \email{guojia@umich.edu}.}
\renewcommand{\thefootnote}{\arabic{footnote}}

\pagenumbering{arabic}

\begin{abstract}\noindent
In this paper we develop two numerical methods for optimal stopping
in the framework of one dimensional diffusion.
Both of the methods use the Skorohod embedding in order to construct
\emph{recombining tree} approximations for diffusions with general coefficients. This technique allows us to determine convergence rates and construct nearly optimal stopping
times which are optimal at the same rate.
Finally, we demonstrate the efficiency of our schemes on several models.
\end{abstract}
\begin{keywords}
American Options, Optimal Stopping, Recombining Trees, Skorokhod Embedding
\end{keywords}
\section{Introduction}\label{sec:1}\setcounter{equation}{0}
It is well known that pricing American options
leads to optimal stopping problems (see \cite{J,K}).
For finite maturity optimal stopping problems, there are no explicit solutions even in the relatively simple framework
where the diffusion process is a standard Brownian motion.
Motivated by the calculations of American options prices
in local volatility models we develop numerical schemes for optimal stopping in the framework
of one dimensional diffusion. In particular, we develop
tree based approximations. In general for non-constant volatility models the nodes of the tree
approximation trees do not recombine and this fact results in an exponential
and thus a computationally explosive tree that cannot be used especially in pricing American options.
We will present two novel ways of constructing recombining trees.

\emph{The first numerical scheme} we propose, a recombining binomial tree (on a uniform time and space lattice), is based on correlated random walks.
A correlated random walk is a generalized random walk in the sense
that the increments are not identically and independently distributed and the one we introduce has one step memory.
The idea to use correlated random
walks for approximating diffusion processes goes back to \cite{GS}, where the authors
studied the weak convergence of correlated random walks to Markov diffusions.
The disadvantage of the weak convergence approach is that it can not provide error
estimates. Moreover,
the weak convergence result can not be applied for numerical computations
of the optimal stopping time. In order to obtain error estimates for the approximations and
calculate numerically the optimal control
we should consider all
the processes
on the same probability space, and so methods based on strong approximation theorems
come into picture. In this paper we apply the Skorokhod embedding
technique for ”small” perturbations of the correlated random walks. Our approach can be seen as an extension of recombining tree approximations from the
case when a stock evolves according to the geometrical
Brownian motion to the case of a more general diffusion evolution.
This particular case required the Skorokhod embedding into the
 Brownian motion (with a constant variance) and it was treated in the more general game options case in \cite{Ki3,Dk}.

Under boundness and Lipschitz conditions on the drift and volatility of the diffusion, we
obtain error estimates of order $O(n^{-1/4})$. Moreover we show how to construct
a stopping time which is optimal up to the same order.
In fact, we consider a more general setup where the diffusion process might have absorbing barriers.
Clearly, most of the local
volatility modes which are used in practice
(for instance, the CEV model \cite{Co} and the CIR model \cite{CIJR}) do not satisfy the above conditions.
Still, by choosing an appropriate absorbing barriers for these models, we can efficiently approximate the original models,
and for the absorbed diffusions, which satisfy the above conditions, we apply our results.

\emph{Our second numerical scheme} is a recombining trinomial tree, which is obtained by directly sampling the original process at suitable chosen random intervals.
In this method we relax the continuity assumption and work with diffusions with measurable coefficients, e.g. the coefficients can be discontinuous, while we keep
the boundedness conditions. (See Section~\ref{sec:jump} for an example).
The main idea is to construct a sequence of random times
such that
the increment of the diffusion between two sequel times belongs to some fixed set of the form
$\left\{-
\bar{\sigma} \sqrt\frac{T}{n},0,\bar{\sigma}\sqrt\frac{T}{n}\right\}$
and the expectation of the difference between two sequel times equals to $\frac{T}{n}$.
Here $n$ is the discretization parameter
and $T$ is the maturity date.
This idea is inspired by the recent work \cite{AKU1,AKU2}
where the authors applied Skorohod embedding in order to obtain
an Euler approximation of irregular one dimensional diffusions.
Following their methodology we construct
an exact scheme along stopping times.
The constructions are different and in particular, in contrast to the above papers
we introduce a recombining tree approximation.
The above papers do provide error estimate (in fact they are of the same order as our error estimates)
for the expected value of a function of the terminal value (from mathematical finance point of view, this can be viewed as European Vanilla options).
Since we deal with American options, our proof requires additional machinery which allows to treat stopping times.

The second method is more direct and does not require the construction of a diffusion perturbation.
In particular it can be used
for the computations of Barrier options prices; see Remark~\ref{sec:remarks}.
As we do for the first method, we obtain error estimates of order $O(n^{-1/4})$ and construct a stopping time which is optimal up to the same order.

\textbf{Comparison with other numerical schemes.}
The most well-known approach to evaluate American options in one dimension is the finite difference scheme called the projected SOR method, see e.g. \cite{MR1357666}. As opposed to a tree based scheme, the finite difference scheme needs to artificially restrict the domain and impose boundary conditions to fill in the entire lattice.
  The usual approach is the so-called \emph{far-field} boundary condition (as is done in  \cite{MR1357666}). See the discussion in  \cite{WZ}. The main problem is that it is not known how far is sufficient and the error is hard to quantify.
This problem was addressed recently by \cite{WZ} for put options written on CEV models by what they call an artificial boundary method, which is model dependent in that an exact boundary
condition is derived that is then embedded into the numerical scheme. (In fact, the PDE is transformed into a new PDE which satisfies a Neumann boundary condition.) This technique requires having an explicit expression of a Sturm-Lioville equation associated with the Laplace transform of the pricing equation.
As an improvement over \cite{WZ}  the papers
\cite{WZ2, 2016arXiv160600530K} propose a numerical method based on solving the integral equations that the optimal exercise boundary satisfies.
 None of these papers discuss the
convergence rate or prove that their proposed optimal exercise times are nearly optimal.  We have the advantage of having an optimal stopping problem at the discrete level and without much additional effort are able to show that these stopping times are actually approximately optimal when applied to the continuous time problem. We compare our numerical results to these papers in Table~\ref{tab:comp}.

Another popular approach is the Euler scheme or other Monte-Carlo based schemes (see e.g. \cite{MR1999614, MR2331321}). The grid structure allows for an efficient numerical
computations of stochastic control problems via dynamical programming since computation of conditional expectation simplifies considerably. Compare this to the least squares
method in \cite{doi:10.1093/rfs/14.1.113} (also see \cite{MR1999614}). Besides the computational complexity, the bias in the least squares is hard to characterize since it also
relies on the choice of bases functions. Recently a hybrid approach was developed in \cite{MR3190342}. (The scope of this paper is more general and is in the spirit
of \cite{doi:10.1093/rfs/14.1.113}). Our discussion above in comparing our proposed method to Monte-Carlo applies here. Moreover, although \cite{MR3190342} provides a convergence
rate analysis using PDE methods (requiring regularity assumptions on the coefficients, which our scheme does not need), our paper in addition can prove that the proposed stopping times from our numerical approximations are nearly optimal.

To summarize, our contributions are
\begin{itemize}
\item We build two recombining tree approximations for local volatility models, which considerably simplifies the computation of conditional expectations.
\item By using the Skorohod embedding technique, we give a proof of the rate of convergence of the schemes.
\item This technique in particular allows us to show that the stopping times constructed from the numerical approximations are nearly optimal and the error is of the same
    magnitude as the convergence rate.

\end{itemize}

This rest of the paper is organized as follows. In the next section we introduce the setup and introduce the two methods we propose and present the main theoretical results of
the paper, namely Theorems~\ref{thm2.1} and \ref{thm2.2} (see the last statements in subsections~\ref{sec:2.1} and \ref{sec:2.2}).
 Sections \ref{sec:3} and \ref{sec:4} are devoted to the proof of these respective results. In Section \ref{sec:5} we provide a detailed numerical analysis for both of our
 methods by applying them to various local volatility models.


\section{Preliminaries and Main Results}\label{sec:2}\setcounter{equation}{0}
\subsection{The Setup}\label{sec:2.0}
Let $\{W_t\}_{t=0}^\infty$
be a standard one dimensional Brownian motion, and let $Y_t$, $t\geq 0$ be a one dimensional diffusion
\begin{equation}\label{2.1}
dY_t=\sigma(Y_t)dW_t+\mu(Y_t) dt, \ \ Y_0=x.
\end{equation}
Let $\{\mathcal F_t\}_{t=0}^T$ be a filtration which satisfies the usual conditions such that
$Y_t$, $t\geq 0$ is an adapted process and $W_t$, $t\geq 0$ is a Brownian motion with respect to this filtration.
We assume that the SDE (\ref{2.1}) has a weak solution that is unique in law.

Introduce an
absorbing barriers $B<C$ where $B\in [-\infty,\infty)$
and $C\in(-\infty,\infty]$.
We assume that $x\in (B,C)$ and look at the absorbed stochastic process
$$X_t=\mathbb{I}_{t<\inf\{s:Y_s\notin(B,C)\}}Y_t+\mathbb{I}_{t\geq\inf\{s:Y_s\notin(B,C)\}}Y_{\inf\{s:Y_s\notin(B,C)\}}, \ \ t\geq 0.$$

Motivated by the valuation of American options prices we study
an optimal stopping problem with maturity date $T$ and a reward
$f(t,X_t)$, $t\in [0,T]$
where $f:[0,T]\times\mathbb R\rightarrow\mathbb R_{+}$ satisfies
\begin{equation}\label{2.function}
|f(t_1,x_1)-f(t_2,x_2)|\leq L \left((1+|x_1|)|t_2-t_1|+|x_2-x_1|\right), \ \ t_1,t_2\in [0,T], \ \ x_1,x_2\in\mathbb R
\end{equation}
for some constant $L$. Clearly, payoffs of the form $f(t,x)=e^{-r t}(x-K)^{+}$ (call options) and
$f(t,x)=e^{-r t}(K-x)^{+}$ (put options) fit in our setup.

The optimal stopping value is given by
\begin{equation}\label{2.valfunc}
V=\sup_{\tau\in\mathcal T_{[0,T]}} \mathbb E[f(\tau,X_{\tau})],
\end{equation}
where $\mathcal T_{[0,T]}$ is the set
of stopping times in the interval $[0,T]$, with respect to
the filtration $\mathcal F_t$, $t\geq 0$.
\subsection{Binomial Approximation Scheme Via Random Walks}
\label{sec:2.1}
In this Section we assume the following.
\begin{asm}\label{asm2.1}
The functions $\mu,\sigma:(B,C)\rightarrow\mathbb R$ are bounded and Lipschitz continuous.
Moreover, $\sigma:(B,C)\rightarrow\mathbb R$ is strictly positive and uniformly bounded away from zero.
\end{asm}

\textbf{Binomial Approximation of the State Process.}
Fix $n\in\mathbb N$ and let $h=h(n):=\frac{T}{n}$ be the time discretization.
Since the meaning is clear we will use
$h$ instead of $h(n)$.
Let
$\xi^{(n)}_1,...,\xi^{(n)}_n\in\{-1,1\}$ be a sequence of random variables.
In the sequel, we
always use the initial data $\xi^{(n)}_0\equiv 1$.
Consider the random walk
$$X^{(n)}_k=x+\sqrt{h}\sum_{i=1}^k \xi^{(n)}_i \ \ k=0,1,...,n$$
 with finite absorbing barriers $B_n<C_n$ where
\begin{eqnarray*}
&B_n=x+\sqrt h\min\{k\in\mathbb Z\cap[-n-1,\infty): x+\sqrt h k>B+h^{1/3}\}\\
&C_n=x+\sqrt h\max\{k\in\mathbb Z\cap(-\infty,n+1]: x+\sqrt h k<C-h^{1/3}\}.
 \end{eqnarray*}
Clearly, the process $\{X^{(n)}_k\}_{k=0}^n$ lies on the grid
$x+\sqrt h\{-n,1-n,...,0,1,...,n\}$ (in fact because of the barriers its lie on a smaller grid).

We aim to construct a probabilistic structure so
the pair $\{X^{(n)}_k,\xi^{(n)}_k\}_{k=0}^n$ forms a Markov chain weakly approximating (with the right time change)
the absorbed diffusion $X$. We look for a predictable (with respect the filtration
generated by $\xi^{(n)}_1,...,\xi^{(n)}_n)$
 process $\alpha^{(n)}=\{\alpha^{(n)}_k\}_{k=0}^n$
 which is uniformly bounded (in space and $n$)
 and a probability measure $\mathbb P_n$ on $\sigma\{\xi^{(n)}_1,...,\xi^{(n)}_n\}$
 such that the perturbation defined by
 \begin{equation}\label{2.3+}
 \hat X^{(n)}_k:=X^{(n)}_k+ \sqrt{h}\alpha^{(n)}_k\xi^{(n)}_k, \ \ k=0,1,...,n
 \end{equation}
  is matching the first two moments of the diffusion $X$.
  Namely, we require that for any $k=1,...,n$ (on the event $B_n<X^{(n)}_{k-1}<C_n$)
  \begin{equation}\label{2.4}
  \mathbb E_n\left(\hat X^{(n)}_{k}-\hat X^{(n)}_{k-1}|\xi^{(n)}_1,...,\xi^{(n)}_{k-1}\right)=h\mu(X^{(n)}_{k-1})+o(h),
  \end{equation}
\begin{equation}\label{2.5}
  \mathbb E_n\left((\hat X^{(n)}_{k}-\hat X^{(n)}_{k-1})^2|\xi^{(n)}_1,...,\xi^{(n)}_{k-1}\right)=h\sigma^2(X^{(n)}_{k-1})+o(h),
  \end{equation}
where we use the standard notation $o(h)$ to denote a
random variable that converge to zero (as $h\downarrow 0$)
a.s.
after dividing by $h$.
We also use the convention $O(h)$ to denote a random variable that is uniformly bounded after
dividing by $h$.

From (\ref{2.3+}) it follows that $X^{(n)}_k-\hat X^{(n)}_k=o(1)$.
Hence the convergence of $X^{(n)}$ is equivalent to the convergence
of $\hat X^{(n)}$.
It remains to solve the equations (\ref{2.4})--(\ref{2.5}).
By applying (\ref{2.3+})--(\ref{2.4})
together with the fact that $\alpha^{(n)}$ is predictable and $(\xi^{(n)})^2\equiv 1$
we get
$$\mathbb E_n\left((\hat X^{(n)}_{k}-
\hat X^{(n)}_{k-1})^2|\xi^{(n)}_1,...,\xi^{(n)}_{k-1}\right)=h(1+2\alpha^{(n)}_k)+
h\left((\alpha^{(n)}_k)^2-(\alpha^{(n)}_{k-1})^2\right)+o(h).$$
This together with (\ref{2.5}) gives that
\begin{equation*}\label{2.6}
\alpha^{(n)}_{k}=\frac{\sigma^2(X^{(n)}_{k-1})-1}{2}, \ \ k=0,...,n
\end{equation*}
is a solution, where we set $X^{(n)}_{-1}\equiv x-\sqrt h$.
Next, (\ref{2.4}) yields that
$$\mathbb E_n\left(\xi^{(n)}_{k}|\xi^{(n)}_1,...,\xi^{(n)}_{k-1}\right)=\frac{\alpha^{(n)}_{k-1}\xi^{(n)}_{k-1}+\sqrt h\mu(X^{(n)}_{k-1})}{1+\alpha^{(n)}_{k}}.$$
Recall that $\xi^{(n)}_{k}\in\{-1,1\}$. We conclude that the probability measure $\mathbb P_n$ is given by (prior to absorbing time)
\begin{equation}\label{2.7}
\mathbb P_n\left(\xi^{(n)}_{k}=\pm 1|\xi^{(n)}_1,...,\xi^{(n)}_{k-1}\right)=\frac{1}{2}\left(1 \pm \frac{\alpha^{(n)}_{k-1}\xi^{(n)}_{k-1}+\sqrt
h\mu(X^{(n)}_{k-1})}{1+\alpha^{(n)}_{k}}\right), \ \ k=1,...,n.
\end{equation}
In view of Assumption \ref{asm2.1} we assume that $n$ is sufficiently
large so $\mathbb P_n$ is indeed a probability measure. Moreover, we notice that
$\alpha^{(n)}_{k-1}=\frac{\sigma^2(X^{(n)}_{k-1}-\sqrt h\xi^{(n)}_{k-1})-1}{2}$. Thus the right hand side of
(\ref{2.7}) is determined by $X^{(n)}_{k-1},\xi^{(n)}_{k-1}$, and so
$\{X^{(n)}_k,\xi^{(n)}_k\}_{k=0}^n$ is indeed a Markov chain.

\textbf{Optimal Stopping Problem on a Binomial Tree.}
For any $n$ denote by $\mathcal T_n$ the (finite) set of all
stopping times with respect to filtration
$\sigma\{\xi^{(n)}_1,...,\xi^{(n)}_k\}$, $k=0,...,n$.
Introduce the optimal stopping value
\begin{equation}\label{2.valfuncn}
V_n:=\max_{\eta\in\mathcal T_n} \mathbb E_n
[f(\eta h, X^{(n)}_{\eta})].
\end{equation}
By using standard dynamical programming for optimal stopping (see \cite{PS} Chapter I)
we can calculate $V_n$ and the rational stopping times by the following backward recursion.
For any $k=0,1,...,n$ denote by $G^{(n)}_k$ all the points on the grid
$x+\sqrt h\{-k,1-k,...,0,1,...,k\}$ which lie in the interval $[B_n,C_n]$.

Define the functions
$$\mathfrak{J}^{(n)}_k:G^{(n)}_k\times \{-1,1\}\rightarrow\mathbb R, \ \ k=0,1,...,n$$
\begin{equation*}
\mathfrak{J}^{(n)}_n(z,y)=f(T,z).
\end{equation*}
For $k=0,1,...,n-1$
\begin{eqnarray*}
\mathfrak{J}^{(n)}_k(z,y)=\max\bigg\{f(k h,z),
\sum_{i=1}^2 \frac{1}{2}\left(1+(-1)^{i} \frac{\alpha' y +\sqrt h\mu(z)}{1+\alpha}\right)
\mathfrak{J}^{(n)}((k+1)h, z+(-1)^{i}\sqrt{h})\bigg\}, \, \mbox{if} \ z\in (B_n,C_n),\nonumber
\end{eqnarray*}
where $\alpha'=\frac{\sigma^2(z-y\sqrt{h})-1}{2}$,
$\alpha=\frac{\sigma^2(z)-1}{2}$,
and
\begin{eqnarray*}
&\mathfrak{J}^{(n)}_k(z,y)=\max_{k\leq m\leq n} f(m h,z)
\quad  \mbox{if} \ z\in \{B_n,C_n\}.\nonumber
\end{eqnarray*}
We get that
\begin{equation*}
V_n=\mathfrak{J}^{(n)}_0(x,1)
\end{equation*}
and the stopping time
\begin{equation}\label{2.rational-st}
\eta^{*}_n=n\wedge\min\left\{k: \mathfrak{J}^{(n)}_k(X^{(n)}_k,\xi^{(n)}_k)=f(kh,X^{(n)}_k)\right\}
\end{equation}
is a rational stopping time. Namely,
\begin{equation*}
V_n= \mathbb E_n[f(\eta^{*}_n h, X^{(n)}_{\eta^{*}_n})].
\end{equation*}

\textbf{Skorohod Embedding.}
Before we formulate our first result let us introduce the Skorokhod embedding
which allows
to consider the random variables
$\{\xi^{(n)}_k, X^{(n)}_k\}_{k=0}^n$
(without changing their joint distribution)
and the
diffusion $X$ on the same probability space.

Set,
$$M_t:=x+\int_{0}^t \sigma(Y_s) dW_s=Y_t-\int_{0}^t \mu(Y_s)ds, \ \ t\geq 0.$$
The main idea is to embed the process
$$\tilde X^{(n)}_k:=X^{(n)}_k+ \sqrt h(\alpha^{(n)}_k\xi^{(n)}_k-\alpha^{(n)}_0\xi^{(n)}_0)
-h\sum_{i=0}^{k-1}\mu(X^{(n)}_{i}), \ \ k=0,1,...,n$$
into the martingale $\{M_t\}_{t=0}^\infty$.
Observe that prior the absorbing time, the process ${\{\tilde X^{(n)}_k\}}_{k=0}^n$ is a martingale with respect to the measure $\mathbb P_n$,
and $\tilde X^{(n)}_0=x$.

Set,
$$\alpha^{(n)}_0=\frac{\sigma^2(x-\sqrt h)-1}{2}, \ \ \theta^{(n)}_0=0, \ \
\xi^{(n)}_0=1, \ \ X^{(n)}_0=x.$$
For $k=0,1,...,n-1$ define by recursion
the following random variables
\begin{equation*}
\alpha^{(n)}_{k+1}=\frac{\sigma^2(X^{(n)}_k)-1}{2},
\end{equation*}
If $X^{(n)}_k\in (B_n,C_n)$ then
\begin{equation}\label{2.10}
\theta^{(n)}_{k+1}=\inf\left\{t>\theta^{(n)}_k:
|M_t-M_{\theta^{(n)}_k}+\sqrt{h}\alpha^{(n)}_k\xi^{(n)}_k+h\mu(X^{(n)}_k)|=
\sqrt{h}(1+\alpha^{(n)}_{k+1})\right\},
\end{equation}
\begin{equation}\label{2.11}
\xi^{(n)}_{k+1}=\mathbb{I}_{\theta^{(n)}_{k+1}<\infty}
sgn \left(M_{\theta^{(n)}_{k+1}}-M_{\theta^{(n)}_k}+\sqrt{h}\alpha^{(n)}_k\xi^{(n)}_k+h\mu(X^{(n)}_k)\right),
\end{equation}
where we put $sgn(z)=1$ for $z>0$ and $=-1$ otherwise,
and
\begin{equation}\label{2.12}
X^{(n)}_{k+1}=X^{(n)}_k+\sqrt{h}\xi^{(n)}_{k+1}.
\end{equation}
If $X^{(n)}_k\notin (B_n,C_n)$ then
$\theta^{(n)}_{k+1}=\theta^{(n)}_k+h$ and
$X^{(n)}_{k+1}=X^{(n)}_k$. Set,
$\Theta_n:=n\wedge\min\{k:X^{(n)}_k\notin (B_n,C_n)\}$ and observe that on the event $k<\Theta_n$,
$\theta^{(n)}_{k+1}$ is the stopping time which corresponds to the
Skorokhod embedding
of the binary random
variable with values in the (random) set
$\{\pm \sqrt{h}(1+\alpha^{(n)}_{k+1})-\sqrt{h}\alpha^{(n)}_k\xi^{(n)}_k-h\mu(X^{(n)}_k\}$,
 into the
martingale $\{M_{t}-M_{\theta^{(n)}_k}\}_{t\geq \theta^{(n)}_k}$. Moreover, the grid structure of $X^{(n)}$ implies that if
$X^{(n)}_k\notin (B_n,C_n)$ then $X^{(n)}_k\in \{B_n,C_n\}$.
\begin{lem}
The stopping times $\{\theta^{(n)}_k\}_{k=0}^n$ have a finite mean
and the random variables $\{\xi^{(n)}_k, X^{(n)}_k\}_{k=0}^n$
satisfy $(\ref{2.7})$.
\end{lem}
\begin{proof}
For sufficiently large $n$ we have
that for any $k<n$
$$-\sqrt{h}(1+\alpha^{(n)}_{k+1})-\sqrt {h}\alpha^{(n)}_k\xi^{(n)}_k-h \mu(X^{(n)}_k)<0
<\sqrt{h}(1+\alpha^{(n)}_{k+1})-\sqrt{h}\alpha^{(n)}_k\xi^{(n)}_k-h\mu(X^{(n)}_k).$$
Thus by using the fact that volatility of the martingale $M$ is bounded away from zero, we conclude
$\mathbb E(\theta^{(n)}_{k+1}-\theta^{(n)}_k)<\infty$ for all $k<n$. Hence the stopping times
$\{\theta^{(n)}_k\}_{k=0}^n$ have a finite mean.

Next, we establish (\ref{2.7}) for the redefined $\{\xi^{(n)}_k, X^{(n)}_k\}_{k=0}^n$.
Fix $k$ and consider the event $k<\Theta_n$.
From (\ref{2.10}) we get
\begin{equation}\label{2.13}
M_{\theta^{(n)}_{k+1}}-M_{\theta^{(n)}_k}=\sqrt h(1+\alpha^{(n)}_{k+1})\xi^{(n)}_{k+1}-\sqrt h\alpha^{(n)}_k\xi^{(n)}_k-h\mu(X^{(n)}_k).
\end{equation}
The stochastic process $$\{M_t-M_{\theta^{(n)}_k}\}_{t=\theta^{(n)}_k}^{\theta^{(n)}_{k+1}}$$ is a bounded martingale and so
$\mathbb E(M_{\theta^{(n)}_{k+1}}-M_{\theta^{(n)}_k}|\mathcal F _{\theta^{(n)}_k})=0$. Hence, from (\ref{2.13})
\begin{equation*}
\mathbb E(\xi^{(n)}_{k+1}|\mathcal F_{\theta^{(n)}_k})=\frac{\alpha^{(n)}_k\xi^{(n)}_k+\sqrt h\mu(X^{(n)}_k)}{1+\alpha^{(n)}_{k+1}}.
\end{equation*}
Since $\xi^{(n)}_{k+1}\in\{-1,1\}$ we
arrive at
\begin{equation*}
\mathbb P(\xi^{(n)}_{k+1}=\pm 1|\mathcal F_{\theta^{(n)}_k})=\frac{1}{2}\left(1 \pm \frac{\alpha^{(n)}_{k}\xi^{(n)}_{k}+\sqrt h\mu(X^{(n)}_{k})}{1+\alpha^{(n)}_{k+1}}\right)
\end{equation*}
and conclude that (the above right hand side is $\sigma\{\xi^{(n)}_1,...,\xi^{(n)}_k\}$ measurable)
(\ref{2.7}) holds true.
\end{proof}

\textbf{The first Main Result.}
\begin{thm}\label{thm2.1}
The values $V$ and $V_n$ defined by (\ref{2.valfunc}) and (\ref{2.valfuncn}), respectively satisfy
\[
|V_n-V|=O(n^{-1/4}).\]
Moreover, if we consider the random variables $\xi^{(n)}_0,...,\xi^{(n)}_n$
defined by (\ref{2.10})--(\ref{2.12})
and denote
$\tau^{*}_n\in\mathcal T_{[0,T]}$ by
$\tau^{*}_n=T\wedge\theta^{(n)}_{\eta^{*}_n}$, in which $\eta^{*}_n$ is from (\ref{2.rational-st}) and $\theta^{(n)}_k$ is from (\ref{2.10}),
then
\[
V-\mathbb E [f(\tau^{*}_n,X_{\tau^{*}_n})]=O(n^{-1/4}).
\]
\end{thm}
\subsection{Trinomial Tree Approximation}\label{sec:2.2}
In this section we relax the Lipschitz continuity requirement and assume the following.
\begin{asm}\label{asm2.2}
The functions $\mu,\sigma,\frac{1}{\sigma}:(B,C)\rightarrow\mathbb R $ are bounded and measurable.
\end{asm}
As our assumption indicates the results in this section apply for diffusions with discontinuous coefficients. See Section~\ref{sec:jump} for a pricing problem for a regime switching volatility example. See  also \cite{MR3626624} for other applications of such models.

\begin{rem}\label{newremark}
From the general theory of one dimensional,
time--homogeneous SDE (see Section 5.5 in \cite{KS})
it follows that if $\sigma,\mu:\mathbb R$ are measurable functions such that $\sigma(z)\neq 0$ for all $z\in\mathbb R$ and
the function
$|\mu(z)|+|\sigma(z)|+|\sigma^{-1}(z)|$ is uniformly bonded, then the SDE (\ref{2.1}) has a unique weak solution.
Since the distribution of $X$ is determined only by the values of $\mu,\sigma$ in the interval $(B,C)$ we
obtain (by letting $\sigma,\mu\equiv 1$ outside of the interval $(B,C)$)
that
Assumption \ref{asm2.2} above is sufficient for an existence and uniqueness in law, of the absorbed diffusion $X$. Clearly, Assumption \ref{asm2.1}
is stronger than Assumption \ref{asm2.2}. A relevant reference here is \cite{MR3626624} which not only considers the existence of weak solutions to SDEs but also of their Malliavin differentiability.
\end{rem}

In this section we assume that
$x,B,C\in\mathbb Q\cup\{-\infty,\infty\}$ (recall that the barriers $B$ and $C$ can take the values $-\infty$ and $\infty$, respectively),
and so for
sufficiently large $n$, we can choose
a constant $\bar\sigma=\bar\sigma(n)>\sup_{y\in (B,C)}|\sigma(y)|+\sqrt h\sup_{y\in (B,C)}|\mu(y)|$
which satisfies
$\frac{C-x}{\bar\sigma\sqrt h},\frac{x-B}{\bar\sigma\sqrt h}\in\mathbb N\cup\{\infty\}$.

\textbf{Trinomial Approximation of the State Process.}
The main idea in this section is to find stopping times
$0=\vartheta^{(n)}_0<\vartheta^{(n)}_1<...<\vartheta^{(n)}_n$, such that
for any $k=0,1,...,n-1$
\begin{eqnarray}\label{idea}
X_{\vartheta^{(n)}_{k+1}}-X_{\vartheta^{(n)}_{k}}\in\{-\bar{\sigma}\sqrt h,0,\bar{\sigma}\sqrt h\} \quad
\mbox{and} \ \ \mathbb E(\vartheta^{(n)}_{k+1}-\vartheta^{(n)}_{k}|\mathcal F_{\vartheta^{(n)}_{k}})=h+O(h^{3/2}).
\end{eqnarray}
In this case the random variables
$\{X_{\vartheta^{(n)}_{k}}\}_{k=0}^n$
lie on the grid $x+\bar{\sigma}\sqrt{h}\{-b_n,1-b_n,...,0,1,$\\
$...,c_n\}$
where $b_n=n\wedge\frac{x-B}{\bar\sigma\sqrt h}$ and
$c_n=n\wedge\frac{C-x}{\bar\sigma\sqrt h}$.
Moreover, we will see that
\[
\max_{0\leq k\leq n}|\vartheta^{(n)}_k-k h|=O(\sqrt h).
\]

\textbf{Skorohod Embedding.}
Next, we describe the construction.
For any initial position $B+\bar{\sigma}\sqrt h\leq X_0\leq C-\bar{\sigma}\sqrt{h}$ and
$A\in [0, \bar{\sigma}\sqrt h]$ consider the stopping times
\begin{eqnarray*}\label{2.102}
&\rho^{X_0}_A=\inf\{t: |X_t-X_0|=A\} \ \ \mbox{and}\\
&\kappa^{X_0}_A=\sum_{i=1}^2  \mathbb I_{X_{\rho^{X_0}_A}=X_0+(-1)^i A}
\inf\{t\geq \rho^{X_0}_A: X_t=X_0 \ \mbox{or} \
X_t=X_0+(-1)^i \bar{\sigma}\sqrt h\}.\nonumber
\end{eqnarray*}
Observe that $X_{\kappa^{X_0}_A}-X_0\in\{-\bar{\sigma}\sqrt h , 0,\bar{\sigma}\sqrt h\}$.
Let us prove the following lemma.
\begin{lem}\label{lem.simple}
There exists a unique $\hat A=\hat A(X_0,n)\in (0,\bar{\sigma}\sqrt h]$ such that $\mathbb E(\kappa^{X_0}_{\hat A})=h$.
\end{lem}
\begin{proof}
Clearly, for any $A_1<A_2$
$\kappa^{X_0}_{A_1}\leq \kappa^{X_0}_{A_2}$ a.s.  Since
$\kappa^{X_0}_{A_1}\neq \kappa^{X_0}_{A_2}$ for $A_1\neq A_2$ we conclude that
the function
$g(A):=\mathbb E (\kappa^{X_0}_{A})$ is strictly increasing which satisfies $g(0)=0$. Thus in order to complete the proof it remains to show
that $g(\bar{\sigma}\sqrt h)\geq h$. Observe that $g(\bar{\sigma} \sqrt h)=\mathbb E (\rho^{X_0}_{\bar{\sigma}\sqrt h}).$
Assume (by contradiction) that $\mathbb E (\rho^{X_0}_{\bar{\sigma}\sqrt h})<h$.
From the It\^{o} isometry and the Jensen Inequality we obtain
\begin{align*}
\bar\sigma \sqrt h=&\mathbb E|X_{\rho^{X_0}_{\bar{\sigma}\sqrt h}}-X_0|\leq
\mathbb E\left|\int_{0}^{\rho^{X_0}_{\bar{\sigma}\sqrt h}}\sigma(X_t) dW_t\right|+
\mathbb E\left|\int_{0}^{\rho^{X_0}_{\bar{\sigma}\sqrt h}}\mu(X_t) dt\right|\\
\leq&\sup_{y\in\mathbb R}|\sigma(y)|\sqrt {\mathbb E (\rho^{X_0}_{\bar{\sigma}\sqrt h})}+
\sup_{y\in\mathbb R}|\mu(y)|\mathbb E (\rho^{X_0}_{\bar{\sigma}\sqrt h})<\bar{\sigma}\sqrt h.
\end{align*}
This clearly a contradiction, and the result follows.
\end{proof}
Next, we recall the theory of exit times of Markov diffusions (see Section 5.5 in \cite{KS}).
Set,
\begin{eqnarray}\label{KS}
&p(y)=\int_{X_0}^ y\exp\left(-2 \int_{X_0}^z \frac{\mu(w)}{\sigma^2(w)}dw\right)dz,\\
&G_{a,b}(y,z)=\frac{\left(p(y\wedge z)-p(a)\right)\left(p(b)-p(y\vee z)\right)}{p(b)-p(a)}, \ \ a\leq y,z\leq b,\nonumber\\
&M_{a,b}(y)=\int_{a}^b \frac{2 G_{a,b}(y,z)}{p'(z)\sigma^2(z)}dz, \ \ y\in [a,b].\nonumber
\end{eqnarray}
Then for any $A\in [0,\bar{\sigma} \sqrt h]$
\begin{align}\label{2.formula1}
\mathbb E(\kappa^{X_0}_A)&=\mathbb E(\rho^{X_0}_A)+
\mathbb P(X_{\rho^{X_0}_A}=X_0+A)M_{X_0,X_0+\bar{\sigma} \sqrt h}(X_0+A)\\
&+\mathbb P(X_{\rho^{X_0}_A}=X_0-A)M_{X_0-\bar{\sigma} \sqrt h,X_0}(X_0-A)\nonumber\\&=
M_{X_0-A,X_0+A}(X_0)+\frac{p(X_0)-p(X_0-A)}{p(X_0+A)-p(X_0-A)}M_{X_0,X_0+\bar{\sigma} \sqrt h}(X_0+A)\nonumber\\
&+\frac{p(X_0+A)-p(X_0)}{p(X_0+A)-p(X_0-A)}M_{X_0-\bar{\sigma} \sqrt h,X_0}(X_0-A)\nonumber
\end{align}
and
\begin{eqnarray*}\label{2.formula2}
&\\
&q^{(1)}(X_0,A):=\mathbb P(X_{\kappa^{X_0}_A}=X_0+\bar{\sigma}
h)=\frac{\left(p(X_0)-p(X_0-A)\right)\left(p(X_0+A)-p(X_0)\right)}{\left(p(X_0+A)-p(X_0-A)\right)\left(p(X_0+\bar{\sigma}\sqrt h)-p(X_0)\right)},\nonumber\\
&q^{(-1)}(X_0,A):=\mathbb P(X_{\kappa^{X_0}_A}=X_0-\bar{\sigma}
h)=\frac{\left(p(X_0+A)-p(X_0)\right)\left(p(X_0)-p(X_0-A)\right)}{\left(p(X_0+A)-p(X_0-A)\right)\left(p(X_0)-p(X_0-\bar{\sigma}\sqrt h)\right)},\nonumber\\
&q^{(0)}(X_0,A):=\mathbb P(X_{\kappa^{X_0}_A}=X_0)=1-q^{(1)}(X_0,A)-q^{(-1)}(X_0,A).\nonumber
\end{eqnarray*}
We aim to find numerically $\hat A=\hat A(X_0,n)\in [0,\bar{\sigma} \sqrt h]$ which satisfies
$\mathbb E(\kappa^{X_0}_{A})=h+O(h^{3/2})$. Observe that $p'(X_0)=1$. From the Mean value theorem and the fact that $\frac{\mu}{\sigma^2}$
is uniformly bonded we obtain that for any $X_0-\bar{\sigma}\sqrt h\leq a\leq y,z\leq b\leq X_0+\bar{\sigma}\sqrt h$
\begin{align*}
\frac{G_{a,b}(y,z)}{p'(z)}=&\frac{\left((1+O(\sqrt h))(y\wedge z-a)\right)\left((1+O(\sqrt h))(b-y\vee z)\right)}{(1+O(\sqrt h))^2(b-a)}\\
=&(1+O(\sqrt h))\frac{(y\wedge z-a)(b-y\vee z)}{(b-a)}.
\end{align*}
Hence, for any $X_0-\bar{\sigma}\sqrt h\leq a\leq y\leq b\leq X_0+\bar{\sigma}\sqrt h$
$$M_{a,b}(y)=2\int_{a}^b\frac{(y\wedge z-a)(b-y\vee z)}{(b-a)\sigma^2(z)}dz+O(h^{3/2}).$$
This together with (\ref{2.formula1}) yields
\begin{align}\label{2.formula3}
\mathbb E(\kappa^{X_0}_A)=&\int_{X_0-A}^{X_0+A}\frac{(X_0\wedge z+A-X_0)(X_0+A-X_0\vee z)}{A\sigma^2(z)}dz\\
+&\int_{X_0}^{X_0+\bar{\sigma}\sqrt h}\frac{((X_0+A)\wedge z-X_0)
(X_0+\bar{\sigma}\sqrt h-(X_0+A)\vee z)}{\bar{\sigma}\sqrt h\sigma^2(z)}dz\nonumber\\
+&\int_{X_0-\bar{\sigma}\sqrt h}^{X_0}
\frac{((X_0-A)\wedge z+\bar{\sigma}\sqrt h-X_0)(X_0-(X_0-A)\vee z)}{\bar{\sigma}\sqrt h\sigma^2(z)}dz+O(h^{3/2}).\nonumber
\end{align}
Thus, $\hat A=\hat{A}(X_0,n)$ can be calculated numerically by applying
the bisection method and (\ref{2.formula3}).
\begin{rem}\label{rem.simple}
If in addition to Assumption \ref{asm2.2} we assume that $\sigma$ is Lipschitz then (\ref{2.formula3}) implies
\begin{align*}
\sigma^2(X_0)\mathbb E(\kappa^{X_0}_A)=&
A^2+2\left(\frac{A^2(\bar{\sigma}\sqrt h-A)+A(\bar{\sigma}\sqrt h-A)^2}{2\bar{\sigma}\sqrt h}\right)+O(h^{3/2})\\
=&\bar{\sigma}A\sqrt h+O(h^{3/2}).
\end{align*}
Thus for the case where $\sigma$ is Lipschitz
we set
\begin{equation*}
\hat A(X_0,n)=\frac{\sigma^2(X_0)}{\bar\sigma}\sqrt h.
\end{equation*}
\end{rem}

Now, we define the Skorokhod embedding by the following recursion.
Set $\vartheta^{(n)}_0=0$ and for $k=0,1,...,n-1$
\begin{equation}\label{2.sstpt}
\vartheta^{(n)}_{k+1}=\mathbb{I}_{X_{\vartheta^{(n)}_k}\in (B,C)} \ \kappa^{X_{\vartheta^{(n)}_k}}_{\hat A(X_{\vartheta^{(n)}_k},n)}
+\mathbb{I}_{X_{\vartheta^{(n)}_k}\notin (B,C)}(\vartheta^{(n)}_k+h).
\end{equation}
From the definition of $\kappa$ and $\hat A(\cdot,n)$
it follows that (\ref{idea}) holds true.

\textbf{Optimal Stopping of the Trinomial Model.}
Denote by $\mathcal S_n$ the set of all stopping times
with respect to the filtration
$\{\sigma(X_{\vartheta^{(n)}_1},...,X_{\vartheta^{(n)}_k})\}_{k=0}^n$,
with values in the set $\{0,1,...,n\}$.
Introduce the corresponding optimal stopping value
\begin{equation}\label{2.valfuncn2}
\tilde V_n:=\max_{\eta\in\mathcal S_n} \mathbb E[f(\eta h,X_{\vartheta^{(n)}_{\eta}})].
\end{equation}
As before, $\tilde V_n$ and the rational stopping times can be found by applying dynamical programming. Thus,
define the functions
$$\mathcal J^{(n)}_k:\{x+\bar{\sigma}\sqrt{h}\{-(k\wedge b_n),1-(k\wedge b_n),...,0,1,...,k\wedge c_n\}\}\rightarrow\mathbb R, \ \ k=0,1,...,n$$
\begin{equation*}
\mathcal J^{(n)}_n(z)=f(T,z).
\end{equation*}
For $k=0,1,...,n-1$
$$
\mathcal{J}^{(n)}_k(z)=\max\left(f(k h,z),\sum_{i=-1,0,1}q^{(i)}(z,\hat A(z,n))
\mathcal J^{(n)}_{k+1} (z+i\bar{\sigma}\sqrt{h})\right)
 \ \mbox{if} \ z\in (B,C)$$
and
$$\mathcal{J}^{(n)}_k(z)=\max_{k\leq m\leq n}f(mh,z) \ \mbox{if} \ z\in \{B,C\}.$$
We get that
\begin{equation*}
\tilde V_n=\mathcal J^{(n)}_0(x)
\end{equation*}
and the stopping times given by
\begin{equation}\label{2.yanst}
\tilde\eta^{*}_n=n\wedge\min\left\{k: \mathcal J^{(n)}_k(X_{\vartheta^{(n)}_k})=f(kh,X_{\vartheta^{(n)}_k})\right\}
\end{equation}
satisfies
\begin{equation*}
\tilde V_n=\mathbb E[f(\tilde\eta^{*}_n h,X_{\vartheta^{(n)}_{\tilde\eta^{*}_n}})].
\end{equation*}

\textbf{The second main result.}

\begin{thm}\label{thm2.2}
The values $V$ and $\tilde V_n$ defined by (\ref{2.valfunc}) and (\ref{2.valfuncn2}) satisfy
\begin{equation*}
|V-\tilde V_n|=O(n^{-1/4}).
\end{equation*}
Moreover, if we denote
$\tilde\tau^{*}_n=T\wedge\vartheta^{(n)}_{\tilde\eta^{*}_n}$,
where $\vartheta^{(n)}$ is defined by (\ref{2.sstpt}) and $\tilde\eta^{*}_n$ by (\ref{2.yanst}),
then
\[
V-\mathbb E [f(\tau^{*}_n,X_{\tau^{*}_n})]=O(n^{-1/4}).
\]
\end{thm}

\begin{rem}
Theorems \ref{thm2.1} and \ref{thm2.2} can be extended with the same error estimates to the setup of
Dynkin games which are corresponding to game options (see \cite{Ki1,Ki2}).
The dynamical programming in discrete time
can be done in a similar way by applying
the results from \cite{O}. Moreover, as in the American options case,
the Skorokhod embedding technique
allows to lift the rational times from the
discrete setup to the continuous one. Since the proof for Dynkin games is very similar to our setup,
then for simplicity, in this work we focus on optimal stopping and the pricing of American options.
\end{rem}
\section{Proof of Theorem \ref{thm2.1}}\label{sec:3}\setcounter{equation}{0}
\begin{proof}
Fix $n\in\mathbb N$.
Recall the definition of
$\theta^{(n)}_k,\mathcal T_n,\eta^{*}_n$ from Section \ref{sec:2.1}.
Denote by $\mathbb T_n$ the set of all stopping
times with respect to the filtration
$\{\mathcal F_{\theta^{(n)}_k}\}_{k=0}^n$, with values in $\{0,1,...,n\}$.
Clearly, $\mathcal T_n\subset\mathbb T_n$.
From the strong Markov property of the diffusion $X$ it follows that
\begin{equation}\label{3.0}
V_n=\sup_{\eta\in\mathbb T_n}\mathbb E [f(\eta_n h,X^{(n)}_{\eta_n})]
=\mathbb E[f(\eta^{*}_n h,X^{(n)}_{\eta^{*}_n})].
\end{equation}
Define the function
$\phi_n:\mathcal T_{[0,T]}\rightarrow\mathbb T_n$
by
$\phi_n(\tau)=n\wedge\min\{k:\theta^{(n)}_k\geq \tau\}.$
The equality (\ref{3.0}) implies that
$V_n\geq \sup_{\tau\in\mathcal T^X_{[0,T]}} \mathbb E[f(\phi_n(\tau) h,X^{(n)}_{\phi_n (\tau)})]$.
Hence, from
(\ref{2.function}) we obtain
\begin{align}\label{3.1}
V&\leq V_n+O(1) \sup_{\tau\in\mathcal T_{[0,T]}} \mathbb E |X_{\tau}-X_{\theta^{(n)}_{\phi_n(\tau)}}|
+O(1)\mathbb E\left(\max_{0\leq i\leq n}|X^{(n)}_i-X_{\theta^{(n)}_i}|\right)\\
&+O(1)\mathbb E\left(\sup_{0\leq t\leq T}X_t
\left(\max_{1\leq i\leq n}|\theta^{(n)}_i-i h|+
\max_{1\leq i\leq n}(\theta^{(n)}_i-\theta^{(n)}_{i-1})\right)\right).\nonumber
\end{align}
Next, recall the definition of $\tau^{*}_n$ and observe that
\begin{align}\label{3.2}
V\geq& \mathbb E[f(\tau^{*}_n,X_{\tau^{*}_n})]\geq V_n-
O(1)\mathbb E\left(\max_{0\leq i\leq n}|X^{(n)}_i-X_{\theta^{(n)}_i}|\right)\\
-&O(1)\sup_{\theta^{(n)}_n\wedge T\leq t\leq \theta^{(n)}_n\vee T}|X_t-X_{\theta^{(n)}_n\wedge T}|-
O(1)\mathbb E\left(\sup_{0\leq t\leq T}X_t
\max_{1\leq i\leq n}|\theta^{(n)}_i-i h|\right).\nonumber
\end{align}
Let $\Theta:=\inf\{t:X_t\notin (B,C)\}$ be the absorbing time.
From the Burkholder--Davis--Gundy inequality and the inequality
$(a+b)^m\leq 2^m (a^m+b^m)$, $a,b\geq 0$, $m>0$
it follows that for any $m>1$
and stopping times $\varsigma_1\leq \varsigma_2$
\begin{equation}\label{BDG}
\begin{split}
\mathbb E\left(\sup_{\varsigma_1\leq t\leq \varsigma_2}|X_t-X_s|^m\right) &\leq 2^m\mathbb E\left(\sup_{\varsigma_1\wedge\Theta\leq t\leq \varsigma_2\wedge\Theta}
|M_t-M_{\varsigma_1\wedge\Theta}|^m+
||\mu||_{\infty}^m(\varsigma_2\wedge\Theta-\varsigma_1\wedge\Theta)^m\right)\\
&=O(1)\mathbb E\left(\left|\int_{\varsigma_1\wedge\Theta}^{\varsigma_2\wedge\Theta} \sigma^2(X_t) dt\right|^{m/2}+(\varsigma_2-\varsigma_1)^m\right)\\
&=O(1)\mathbb E \left((\varsigma_2-\varsigma_1)^{m/2}+(\varsigma_2-\varsigma_1)^m\right).
\end{split}
\end{equation}
Observe that
for any stopping time $\tau\in\mathcal T_{[0,T]}$,
$|\tau-\theta^{(n)}_{\phi_n(\tau)}|\leq \max_{1\leq i\leq n}(\theta^{(n)}_i-\theta^{(n)}_{i-1})+|T-\theta^{(n)}_n|$. Moreover,
$2\max_{1\leq i\leq n}|\theta^{(n)}_i-i h|+h\geq \max_{1\leq i\leq n}(\theta^{(n)}_i-\theta^{(n)}_{i-1})$.
Thus Theorem \ref{thm2.1} follows from (\ref{3.1})--(\ref{BDG}), the Cauchy--Schwarz inequality, the Jensen inequality
and Lemmas \ref{lem3.1}--\ref{lem3.2} below.
\end{proof}

\subsection{Technical estimates for the proof of Theorem \ref{thm2.1}}

The next lemma is a technical step in proving Lemmas \ref{lem3.1}--\ref{lem3.2}
which are the main results of this subsection, which are then used for the proof of Theorem \ref{thm2.1}.
\begin{lem}\label{important}
Recall the definition of $\Theta_n$ given after (\ref{2.12}). For any $m>0$
$$\mathbb E\left(\max_{0\leq k\leq\Theta_n}|Y_{\theta^{(n)}_k}-X^{(n)}_k|^m\right)=O(h^{m/2}).$$
\end{lem}
\begin{proof}
From the Jensen inequality it follows that it is sufficient to prove the claim
for $m>2$. Fix $m>2$ and $n\in\mathbb N$. We will apply the discrete version of the Gronwall inequality.
 Introduce the random variables $U_k:=\mathbb{I}_{k\leq\Theta_n}|X^{(n)}_{k}-Y_{\theta^{(n)}_k}|$.
Since the process $\alpha^{(n)}$ is uniformly bounded and $\mu$ is Lipschitz continuous,
then from (\ref{2.3+}) and (\ref{2.13}) we obtain
that
\begin{align}\label{3.9+}
U_k=& O(\sqrt h)+ \mathbb{I}_{k\leq\Theta_n}\left|h \sum_{i=0}^{k-1} \mu(X^{(n)}_i)-\sum_{i=0}^{k-1}
\int_{\theta^{(n)}_{i}}^{\theta^{(n)}_{i+1}}\mu(Y_t)dt\right|\\
\leq&O(\sqrt h)+ \mathbb{I}_{k\leq\Theta_n}\left|h \sum_{i=0}^{k-1} \mu(X^{(n)}_i)-\sum_{i=0}^{k-1}
\mu(Y_{\theta^{(n)}_i})(\theta^{(n)}_{i+1}-\theta^{(n)}_i)\right|\nonumber\\
+&O(1)\sum_{i=0}^{k-1}\mathbb I_{i<\Theta_n}\sup_{\theta^{(n)}_i\leq t\leq \theta^{(n)}_{i+1}}
|Y_t-Y_{\theta^{(n)}_i}|(\theta^{(n)}_{i+1}-\theta^{(n)}_i)\nonumber\\
\leq&O(\sqrt h)+O(1)\sum_{i=0}^{k-1}L_i+
O(h)\sum_{i=0}^{k-1}U_i+|\sum_{i=0}^{k-1}(I_i+J_i)| \nonumber
\end{align}
where
\begin{align}
I_i:&=\mathbb{I}_{i<\Theta_n}\mu(Y_{\theta^{(n)}_i})\left(\theta^{(n)}_{i+1}-\theta^{(n)}_i-\mathbb E(\theta^{(n)}_{i+1}-\theta^{(n)}_i|\mathcal F_{\theta^{(n)}_i})\right),
\nonumber \\
J_i:&=\mathbb{I}_{i<\Theta_n}\mu(Y_{\theta^{(n)}_i})\left(\mathbb E(\theta^{(n)}_{i+1}-\theta^{(n)}_i|\mathcal F_{\theta^{(n)}_i})-h\right), \nonumber  \\
L_i:&=\mathbb{I}_{i<\Theta_n}\sup_{\theta^{(n)}_i\leq t\leq \theta^{(n)}_{i+1}}|Y_t-Y_{\theta^{(n)}_i}|(\theta^{(n)}_{i+1}-\theta^{(n)}_i). \nonumber
\end{align}
Next, from (\ref{2.7}), (\ref{2.13}) and the It\^{o} Isometry
it follows that on the event $i<\Theta_n$
\begin{align}\label{3.14}
\mathbb E &
\left(\int_{\theta^{(n)}_i}^{\theta^{(n)}_{i+1}}\sigma^2(Y_t)dt|\mathcal F_{\theta^{(n)}_i}\right)=
\mathbb E\left((M_{\theta^{(n)}_{i+1}}-M_{\theta^{(n)}_i})^2 |\mathcal F_{\theta^{(n)}_i}\right)\\
=&h(1+2\alpha^{(n)}_{i+1})+h\left((\alpha^{(n)}_{i+1})^2-(\alpha^{(n)}_{i})^2\right)+O(h^{3/2})
=
h\sigma^2(X^{(n)}_i)+O(h^{3/2}).\nonumber
\end{align}
On the other hand, the function $\sigma$ is bounded and Lipschitz, and so
\begin{align}\label{3.14+}
\mathbb E\left(\int_{\theta^{(n)}_i}^{\theta^{(n)}_{i+1}}\sigma^2(Y_t)dt|\mathcal F_{\theta^{(n)}_i}\right)=&
\sigma^2(X^{(n)}_i)\mathbb E(\theta^{(n)}_{i+1}-\theta^{(n)}_i|\mathcal F_{\theta^{(n)}_i})\\
+&O(1)\mathbb E\left(\sup_{\theta^{(n)}_i\leq t\leq \theta^{(n)}_{i+1}}
|Y_t-Y_{\theta^{(n)}_i}|(\theta^{(n)}_{i+1}-\theta^{(n)}_i)|\mathcal F_{\theta^{(n)}_i}\right)
\nonumber\\
+&O(1)\left|X^{(n)}_i-Y_{\theta^{(n)}_i}\right|\mathbb E\left(\theta^{(n)}_{i+1}-\theta^{(n)}_i|\mathcal F_{\theta^{(n)}_i}\right).\nonumber
\end{align}
From (\ref{3.14})--(\ref{3.14+}) and the fact that $\sigma$ bounded away from zero we get
that on the event $i<\Theta_n$
\begin{eqnarray}\label{3.15}
&\mathbb E&(\theta^{(n)}_{i+1}-\theta^{(n)}_i|\mathcal F_{\theta^{(n)}_i})
\\ &=&h+O(h^{3/2})+O(1)\mathbb E\left(\sup_{\theta^{(n)}_i\leq t\leq \theta^{(n)}_{i+1}}
|Y_t-Y_{\theta^{(n)}_i}|(\theta^{(n)}_{i+1}-\theta^{(n)}_i)|\mathcal F_{\theta^{(n)}_i}\right) \nonumber \\
&+&O(1)\left|X^{(n)}_i-Y_{\theta^{(n)}_i}\right|\mathbb E\left(\theta^{(n)}_{i+1}-\theta^{(n)}_i|\mathcal F_{\theta^{(n)}_i}\right).\nonumber
\end{eqnarray}
Clearly, (\ref{3.14}) implies that ($\sigma$ is bounded away from zero)
on the event $i<\Theta_n$,
$\mathbb E(\theta^{(n)}_{i+1}-\theta^{(n)}_i|\mathcal F_{\theta^{(n)}_i})=O(h).$
This together with (\ref{3.15}) gives ($\mu$ is bounded)
\begin{equation}\label{3.16}
|J_i|= O(h^{3/2})+O(h)U_i+O(1)
\mathbb E(L_i|\mathcal F_{\theta^{(n)}_i}).
\end{equation}
From (\ref{3.9+}) and (\ref{3.16}) we obtain that for any $k=1,...,n$
\begin{eqnarray*}
\max_{0\leq j\leq k}U_j&=&
O(\sqrt h)+ O(h)
\sum_{i=0}^{k-1}\max_{0\leq j\leq i}U_j \\
&+&\max_{0\leq k\leq n-1}|\sum_{i=0}^ k I_i|+ O(1)
\sum_{i=0}^{n-1} L_i+O(1)
\sum_{i=0}^{n-1} \mathbb E(L_i|\mathcal F_{\theta^{(n)}_i}).
\end{eqnarray*}
Next, recall the following inequality, which is a direct consequence of the Jensen's inequality,
\begin{equation}\label{convex}
(\sum_{i=1}^n a_i)^{\tilde m}\leq n^{\tilde m-1}\sum_{i=1}^n a^{\tilde m}_i,
 \  \
a_1,...,a_n\geq 0, \  \tilde m\geq 1.
\end{equation}
Using the above inequality along with Jensen's inequality we arrive at
\begin{eqnarray*}
\mathbb E\left(\max_{0\leq j\leq k}U^m_j\right)&=&O(h^{m/2})+ O(h) \sum_{i=0}^{k-1}
\mathbb E\left(\max_{0\leq j\leq i}U^m_j\right)\\&+&O(1)\mathbb E\left(\max_{0\leq k\leq n-1}|\sum_{i=0}^ k I_i|^m\right)+ O(n^{m-1})
\sum_{i=0}^{n-1} \mathbb E(L^m_i).\nonumber
\end{eqnarray*}
From the discrete version of Gronwall's inequality (see \cite{C})
\begin{eqnarray}\label{3.17}
&\mathbb E&\left(\max_{0\leq i\leq n}U^m_i\right)
 \\ &=&(1+ O(h))^n\times
\left(O(h^{m/2})+O(1)\mathbb E\left(\max_{0\leq k\leq n-1}|\sum_{i=0}^ k I_i|^m\right)+ O(n^{m-1})
\sum_{i=0}^{n-1} \mathbb E(L^m_i)\right) \nonumber \\
&=&O(h^{m/2})+O(1)\mathbb E\left(\max_{0\leq k\leq n-1}|\sum_{i=0}^ k I_i|^m\right)+ O(n^{m-1})
\sum_{i=0}^{n-1} \mathbb E(L^m_i).\nonumber
\end{eqnarray}
Next, we estimate
$\mathbb E\left(\max_{0\leq k\leq n-1}|\sum_{i=0}^ k I_i|^m\right)$
and $\mathbb E (L^m_i|\mathcal F_{\theta^{(n)}_i})$, $i=0,1,...,n-1.$
By applying the Burkholder–-Davis–-Gundy inequality
for the martingale $\{M_t-M_{\theta^{(n)}_i}\}_{t=\theta^{(n)}_i}^{\theta^{(n)}_{i+1}}$
it follows that for any $\tilde m>1/2$
\begin{align*}
\mathbb I_{i<\Theta_n}&\mathbb E\left(\left(\int_{\theta^{(n)}_i}^{\theta^{(n)}_{i+1}}\sigma^2(Y_t)dt\right)^{\tilde m}\bigg|\mathcal F_{\theta^{(n)}_i}\right)\\=&
O(1)\mathbb I_{i<\Theta_n}
\mathbb  E\left(\max_{\theta^{(n)}_i\leq t\leq \theta^{(n)}_{i+1}}\left(M_t-M_{\theta^{(n)}_i}\right)^{2\tilde m}
\bigg|\mathcal F_{\theta^{(n)}_i}\right)=O(h^{\tilde m})\mathbb I_{i<\Theta_n}
\end{align*}
where the last equality follows from
the fact that $$\mathbb I_{i<\Theta_n}\max_{\theta^{(n)}_i\leq t\leq \theta^{(n)}_{i+1}}|M_t-M_{\theta^{(n)}_i}|=
O(\sqrt h)\mathbb I_{i<\Theta_n}.$$
Since $\sigma$ is bounded away from zero we get
\begin{equation}\label{3.18}
\mathbb I_{i<\Theta_n}\mathbb E\left((\theta^{(n)}_{i+1}-\theta^{(n)}_i)^{\tilde m}|\mathcal F_{\theta^{(n)}_i}\right)
=O(h^{\tilde m})\mathbb I_{i<\Theta_n}, \ \ \tilde m>1/2.
\end{equation}
Next, observe that $\sum_{i=0}^k I_i$, $k=0,...,n-1$ is a martingale. From
the Burkholder��-Davis��-Gundy inequality, (\ref{convex}), (\ref{3.18}) and the fact that $\mu$ is bounded we conclude
\begin{eqnarray}\label{3.19}
\mathbb  E\left(\max_{0\leq k\leq n-1}|\sum_{i=0}^ k I_i|^m\right)&\leq& O(1)\mathbb E
 \left(\left(\sum_{i=0}^{n-1} I^2_i\right)^{m/2} \right)\\
&\leq& O(1) n^{m/2-1}\sum_{i=0}^{n-1}\mathbb E[|I_i|^m]=O(1)n^{m/2-1} n O(h^{m}) \nonumber
\\&=&O(h^{m/2}).\nonumber
\end{eqnarray}
Finally, we estimate $\mathbb E(L^m_i|\mathcal F_{\theta^{(n)}_i})$ for $i=0,1,...,n-1$. Clearly, on the event $i<\Theta_n$,
\begin{align*}
\sup_{\theta^{(n)}_i\leq t\leq \theta^{(n)}_{i+1}}|Y_t-Y_{\theta^{(n)}_i}|
\leq&\sup_{\theta^{(n)}_i\leq t\leq \theta^{(n)}_{i+1}}|M_t-M_{\theta^{(n)}_i}|+||\mu||_{\infty}(\theta^{(n)}_{i+1}-\theta^{(n)}_i)\\
=&
O(\sqrt h)+||\mu||_{\infty}(\theta^{(n)}_{i+1}-\theta^{(n)}_i).
\end{align*}
Hence, from (\ref{3.18})
we get
\begin{eqnarray}\label{3.20-}
\mathbb E(L^m_i|\mathcal F_{\theta^{(n)}_i})
&\leq& \mathbb I_{i<\Theta_n} \mathbb E
\left(O(h^{m/2})(\theta^{(n)}_{i+1}-\theta^{(n)}_i)^m+O(1)(\theta^{(n)}_{i+1}-\theta^{(n)}_i)^{2m}|\mathcal F_{\theta^{(n)}_i}\right)\\ &=&O(h^{3m/2}). \nonumber
\end{eqnarray}
This together with
(\ref{3.17}) and (\ref{3.19}) yields
$\mathbb E\left(\max_{0\leq i\leq n}U^m_i\right)=O(h^{m/2})$, and the result follows.
\end{proof}

\begin{lem}\label{lem3.1}
For any $m>0$
$$\mathbb E\left(\max_{1\leq k\leq n}|\theta^{(n)}_k- kh|^m\right)=O(h^{m/2}).$$
\end{lem}
\begin{proof}
Fix $m>2$ and $n\in\mathbb N$. Observe (recall the definition after (\ref{2.12})) that for $k\geq\Theta_n$ we have
$\theta^{(n)}_k-kh=\theta^{(n)}_{\Theta_n}-\Theta_n h$. Hence
\begin{equation}\label{3.new}
\max_{1\leq k\leq n}|\theta^{(n)}_k- kh|=\max_{1\leq k\leq \Theta_n}|\theta^{(n)}_k- kh|.
\end{equation}
Redefine the terms $I_i, J_i$ from Lemma \ref{important} as following
\begin{equation*}
I_i:=\mathbb{I}_{i<\Theta_n}\left(\theta^{(n)}_{i+1}-\theta^{(n)}_i-\mathbb E(\theta^{(n)}_{i+1}-\theta^{(n)}_i|\mathcal F_{\theta^{(n)}_i})\right),
\end{equation*}
\begin{equation*}
J_i:=\mathbb{I}_{i<\Theta_n}\left(\mathbb E(\theta^{(n)}_{i+1}-\theta^{(n)}_i|\mathcal F_{\theta^{(n)}_i})-h\right).
\end{equation*}
Then similarly to (\ref{3.18})--(\ref{3.19}) it follows that
$$\mathbb E\left(\max_{0\leq k\leq n-1}|\sum_{i=0}^ k I_i|^m\right)=O(h^{m/2}).$$
Moreover, by applying Lemma \ref{important}, (\ref{3.16}) and (\ref{3.20-})
we obtain that for any $i$, $\mathbb E[|J_i|^m]=O(h^{3m/2})$.
We conclude that
\begin{eqnarray*}
\mathbb E\left(\max_{1\leq k\leq \Theta_n}|\theta^{(n)}_k- kh|^m\right)&=&O(1)\mathbb E\left(\max_{0\leq k\leq n-1}|\sum_{i=0}^ k I_i|^m\right)+
O(1)\mathbb E\left(\max_{0\leq k\leq n-1}|\sum_{i=0}^ k J_i|^m\right)  \nonumber
\\ &=&O(h^{m/2}).
\end{eqnarray*}
This together with (\ref{3.new}) completes the proof.
\end{proof}
We end this section with establishing the following estimate.
\begin{lem}\label{lem3.2}
$$\mathbb E\left(\max_{0\leq k\leq n}|X_{\theta^{(n)}_k}-X^{(n)}_k|\right)=O(h^{1/4}).$$
\end{lem}
\begin{proof}
Set $\Gamma_n=\sup_{0\leq t\leq\theta^{(n)}_n}|Y_t|+\max_{0\leq k\leq n}|X^{(n)}_k|$, $n\in\mathbb N$.
From Lemma \ref{lem3.1} it follows that
for any $m>1$, $\sup_{n\in\mathbb N}\mathbb E[(\theta^{(n)}_n)^m]<\infty$. Thus, from the Burkholder--Davis--Gundy inequality
and the fact that $\mu,\sigma$ are bounded we obtain
$$\sup_{n\in\mathbb N}\mathbb E[\sup_{0\leq t\leq\theta^{(n)}_n}|Y_t|^m]<\infty.$$
This together with Lemma \ref{important} gives that (recall that $X^{(n)}$ remains constant after $\Theta_n$)
\begin{equation}\label{moments}
\sup_{n\in\mathbb N}\mathbb E[\Gamma^m_n]<\infty \ \ \forall m>1.
\end{equation}
We start with estimating
$\mathbb E\left(\max_{0\leq k\leq \Theta_n}|X_{\theta^{(n)}_k}-X^{(n)}_k|\right)$.
Fix $n$ and introduce the events
\begin{align}
O&=\{\max_{0\leq k< \Theta_n}\sup_{\theta^{(n)}_k\leq t \leq\theta^{(n)}_{k+1}}|Y_t-X^{(n)}_k|\geq  n^{-1/3}\}, \nonumber\\
O_1&=\{\max_{0\leq k< \Theta_n}\sup_{\theta^{(n)}_k\leq t \leq\theta^{(n)}_{k+1}}|Y_t-Y_{\theta^{(n)}_k}|\geq n^{-1/3}/2\}, \nonumber\\
O_2&=\{\max_{1\leq k\leq \Theta_n}|X^{(n)}_k-Y_{\theta^{(n)}_k}|\geq n^{-1/3}/2\}. \nonumber
\end{align}
From Lemma \ref{important} (for $m=6$) and the Markov inequality we get
\begin{equation*}
\mathbb P(O_2)=O(h).
\end{equation*}
Observe that on the event $k<\Theta_n$,
$\sup_{\theta^{(n)}_k\leq t \leq\theta^{(n)}_{k+1}}|M_t-M_{\theta^{(n)}_k}|=O(\sqrt h)$. Hence,
from the fact that $\mu$ is uniformly bounded we obtain that
for sufficiently large $n$
\begin{equation*}
\mathbb P(O_1)\leq \sum_{i=0}^{n-1} \mathbb P\left(\mathbb{I}_{i<\Theta_n}(\theta^{(n)}_{i+1}-\theta^{(n)}_i)>h^{1/2}\right)\leq O(n) h^4/h^2=O(h)
\end{equation*}
where the last inequality follows form the Markov ineqaulity and
(\ref{3.18}) for $\tilde m=4$.
Clearly,
\begin{equation}\label{newformula}
\mathbb P(O)\leq \mathbb P(O_1)+\mathbb P(O_2)=O(h).
\end{equation}
From the simple inequalities
$B_n-B,C-C_n\geq n^{-1/3}$ it follows that
$\{\exists   k\leq\Theta_n: X_{\theta^{(n)}_k}\neq Y_{\theta^{(n)}_k}\}\subset O$.
Thus from
Lemma \ref{important}, (\ref{moments})--(\ref{newformula}) and the Cauchy--Schwarz
inequality it follows
\begin{eqnarray*}
\mathbb E\left(\max_{0\leq k\leq \Theta_n}|X_{\theta^{(n)}_k}-X^{(n)}_k|\right)&\leq&
\mathbb E\left(\max_{0\leq k\leq \Theta_n}|Y_{\theta^{(n)}_k}-X^{(n)}_k|\right)+
\mathbb E[2\Gamma_n\mathbb I_{O}] \\
&\leq& O(\sqrt h)+ O(1)\sqrt {\mathbb P(O)}=O(\sqrt h).
\end{eqnarray*}
It remains to estimate
$\mathbb E\left(\max_{\Theta_n<k\leq n}|X_{\theta^{(n)}_k}-X^{(n)}_k|\right)$.
Let $\mathbb Q\sim \mathbb P$ be the probability measure given by
$$Z_t:=\frac{d \mathbb Q}{d\mathbb P}{|\mathcal F_t}=\exp\left(-\int_{0}^t\frac{\mu(Y_u)}{\sigma(Y_u)}dW_u-\int_{0}^t\frac{\mu^2(Y_u)}{2\sigma^2(Y_u)}du\right), \ \ t\geq 0.$$
From Girsanov's theorem it follows that under the measure $\mathbb Q$,
$Y$ is a martingale. Assume that the martingale $Y$ starts at some $Y_0\in (B,C)$ and we want to estimate
$\mathbb Q(\max_{0\leq t\leq T}X_t-Y_0>\epsilon)$ for $\epsilon>0$.
Define a stopping time
$\varsigma=T\wedge \inf\{t: Y_t=Y_0+\epsilon\}\wedge\inf\{t: Y_t=B\}$. Then from the Optional stopping theorem
$$Y_0=\mathbb E_{\mathbb Q} [Y_{\varsigma}]\geq B(1-\mathbb Q(Y_{\varsigma}=Y_0+\epsilon))+(Y_0+\epsilon)\mathbb Q(Y_{\varsigma}=Y_0+\epsilon).$$
Hence, ($B$ is an absorbing barrier for $X$)
\begin{equation}\label{3.101}
\mathbb Q(\max_{0\leq t\leq T}X_t-Y_0>\epsilon)\leq\mathbb Q(Y_{\varsigma}=Y_0+\epsilon)\leq\frac{Y_0-B}{Y_0+\epsilon-B}.
\end{equation}
Similarly,
\begin{equation}\label{3.102}
\mathbb Q(\max_{0\leq t\leq T}Y_0-X_t>\epsilon)\leq\frac{C-Y_0}{C+\epsilon-Y_0}.
\end{equation}
Next, recall the event $O$ from the beginning of the proof.
Consider the event
$\tilde O=\{\Theta_n<n\}\setminus O$ and choose $\epsilon>3 n^{-1/3}$.
Observe that $\theta^{(n)}_n-\theta^{(n)}_{\Theta_n}=(n-\Theta_n)h\leq T$. Moreover on the event
$\tilde O$, $\min(C-X_{\theta^{(n)}_{\Theta_n}},X_{\theta^{(n)}_{\Theta_n}}-B)\leq 3n^{-1/3}$ (for sufficiently large $n$).
Thus,
from (\ref{3.101})--(\ref{3.102})
\begin{equation}\label{3.102+}
\mathbb Q\left(\tilde O\bigcap\left(\sup_{\theta^{(n)}_{\Theta_n}<t\leq\theta^{(n)}_n}|X_{\theta^{(n)}_{\Theta_n}}-X_t|\geq\epsilon\right)
\bigg|\mathcal{F}_{\theta^{(n)}_{\Theta_n}}\right) \leq
\frac{3n^{-1/3}}{3 n^{-1/3}+\epsilon}  \ \ \mbox{a.s.}
\end{equation}
We conclude that
\begin{eqnarray}\label{3.103}
& \quad \mathbb E_{\mathbb Q} \left(\mathbb{I}_{\tilde 0}\left(1\wedge\sup_{\theta^{(n)}_{\Theta_n}<t\leq\theta^{(n)}_n}
|X_{\theta^{(n)}_{\Theta_n}}-X_t|\right)\bigg||\mathcal{F}_{\theta^{(n)}_{\Theta_n}}\right) \leq\\
&3 n^{-1/3}+\int_{3 n^{-1/3}}^1 \frac{3n^{-1/3}}{3 n^{-1/3}+\epsilon}d\epsilon=O( n^{-1/3}\ln n)  \ \ \mbox{a.s.},\nonumber
\end{eqnarray}
where $\mathbb E_{\mathbb Q}$ denotes the expectation with respect to $\mathbb Q$.

Next, denote $\mathcal Z_n=\frac{Z_{\theta^{(n)}_n}}{Z_{\theta^{(n)}_{\Theta_n}}}$.
Observe that for a random variable $\mathcal X$,
\[\mathbb E(\mathcal X|\mathcal{F}_{\theta^{(n)}_{\Theta_n}})=\mathbb E_{\mathbb Q}(\mathcal Z_m\mathcal X|\mathcal{F}_{\theta^{(n)}_{\Theta_n}}).\]
From the relations
$\theta^{(n)}_n-\theta^{(n)}_{\Theta_n}=(n-\Theta_n)h\leq T$ it follows that for any $m\in\mathbb R$,
$\mathbb E_Q[\mathcal Z^m_n]$ is uniformly bounded (in $n$).
This together with the (conditional)
Holder inequality,
(\ref{moments})--(\ref{newformula}), and
(\ref{3.102+})--(\ref{3.103}) gives
\begin{eqnarray*}
\mathbb E\left(\sup_{\Theta_n<k\leq n}|X_{\theta^{(n)}_k}-X^{(n)}_k|\right) &\leq& \mathbb E[2\Gamma_n\mathbb{I}_O]
+\mathbb E\left(\mathbb E_{\mathbb Q}\left( 2\mathcal Z_n\Gamma_n \mathbb{I}_{\tilde
O}\mathbb{I}_{\sup_{\theta^{(n)}_{\Theta_n}<t\leq\theta^{(n)}_n}|X_{\theta^{(n)}_{\Theta_n}}-X_t|>1}\bigg|\mathcal{F}_{\theta^{(n)}_{\Theta_n}}\right)\right)
\\
&+&\mathbb E\left(\mathbb E_{\mathbb Q}\left(\mathcal Z_n\mathbb{I}_{\tilde O}\left(1\wedge\sup_{\theta^{(n)}_{\Theta_n}<t\leq\theta^{(n)}_n}
|X_{\theta^{(n)}_{\Theta_n}}-X_t|\right)\bigg|\mathcal{F}_{\theta^{(n)}_{\Theta_n}}\right)\right)\\
&+&\mathbb E\left(\mathbb{I}_{\tilde O}|X^{(n)}_{\Theta_n}-X_{\theta^{(n)}_{\Theta_n}}|\right)
\leq O(\sqrt h)+\mathbb E\left(O(1)\left(\frac{3n^{-1/3}}{3 n^{-1/3}+1}\right)^{4/5}\right)\\
&+&\mathbb E\left(O(1)\left(O(n^{-1/3}\ln n)\right)^{4/5}\right)+O(n^{-1/3})=O(h^{1/4})
\end{eqnarray*}
and the result follows.
\end{proof}

\section{Proof of Theorem \ref{thm2.2}}\label{sec:4}\setcounter{equation}{0}
The proof of Theorem \ref{thm2.2} is less technical than the proof of
Theorem \ref{thm2.1}.\\
We start with the following Lemma.
\begin{lem}\label{lem4.1}
Recall the stopping time $\rho^{X_0}_A$ from Section \ref{sec:2.2}. For $X_0\in (B,C)$ and $0<\epsilon<\min(X_0-B,C-X_0)$ we have
$$\mathbb E\left((\rho^{X_0}_{\epsilon})^m\right)=O(\epsilon^{2m}) \ \ \forall m>0.$$
\end{lem}
\begin{proof}
Choose $m>1$, $X_0\in (B,C)$ and $0<\epsilon<\min(X_0-B,C-X_0)$.
Define the stochastic process
$\mathcal M_t=p(Y_t)$, $t\geq 0$ where recall $p$ is given in (\ref{KS}).
It is easy to see that $p$ is strictly increasing function and $\{\mathcal M_t\}_{t\geq 0}$ is a local--martingale
which satisfies $\mathcal M_t=\int_{0}^t p'(Y_u)\sigma(Y_u) dW_u.$
Hence
$\rho^{X_0}_{\epsilon}=\inf\{t:\mathcal M_t=p(X_0\pm \epsilon)\}$.
From the Burkholder--Davis--Gundy inequality it follows that there exists a constant $C>0$ such that
\begin{eqnarray}\label{4.newform}
\max\left\{|p(X_0+\epsilon)|^{2m},|p(X_0-\epsilon)|^{2m}\right\} &\geq&
\mathbb E\left(\max_{0\leq t\leq\rho^{X_0}_{\epsilon}}|\mathcal M_{t}|^{2m}\right)\\
&\geq& C\mathbb E\left(\left(\int_{0}^{\rho^{X_0}_{\epsilon}} |p'(Y_u)\sigma(Y_u)|^2 du\right)^m\right)\nonumber\\
&\geq& C\left(\inf_{y\in[X_0-\epsilon,X_0+\epsilon]}|p'(y)\sigma(y)|\right)^{2m}\mathbb E\left((\rho^{X_0}_{\epsilon})^{m}
\right).\nonumber
\end{eqnarray}
Since $\frac{\mu}{\sigma^2}$ is uniformly bounded, then
for $y\in [X_0-\epsilon,X_0+\epsilon]$ we have
$p(y)=O(\epsilon)$ and $p'(y)=1-O(\epsilon)$. This together with (\ref{4.newform}) competes the proof.
 \end{proof}

Now we are ready to prove Theorem \ref{thm2.2}.

\textbf{Proof of Theorem 2.2.}
The proof follows the steps of the proof of Theorem \ref{thm2.1},
however the current case is much less technical. The reason is that we do not take a perturbation of the process $X_t$, $t\geq 0$
and so we have an exact scheme at the random times
 $\{\vartheta^{(n)}_k\}_{k=0}^n$. Thus,
 we can write similar inequalities to
 (\ref{3.1})--(\ref{3.2}), only now the analogous term to
 $\max_{0\leq i\leq n}|X^{(n)}_i-X_{\theta^{(n)}_i}|$ is vanishing.
Hence, we can skip Lemmas \ref{important} and \ref{lem3.2}.
Namely, in order to complete the proof
of Theorem \ref{thm2.2} it remains to show the following analog of Lemma \ref{lem3.1}:
\begin{equation}\label{4.2}
\mathbb E\left(\max_{1\leq k\leq n}|\vartheta^{(n)}_k-k h|^m\right)=O(h^{m/2}), \ \ \forall m>0.
\end{equation}
Without loff of generality, we assume that $m>1$. Set,
$$H_i=\vartheta^{(n)}_{i}-\vartheta^{(n)}_{i-1}-
\mathbb E\left(\vartheta^{(n)}_{i}-\vartheta^{(n)}_{i-1}|\mathcal F_{\vartheta^{(n)}_{i-1}}\right), \ \ i=1,...,n.$$
Clearly the stochastic process $\{\sum_{i=1}^k H_i\}_{k=1}^n$ is a martingale, and from
(\ref{idea}) we get that for all $k$,
$\vartheta^{(n)}_k-kh=\sum_{i=1}^k H_i+ O(\sqrt h)$.
Moreover, on the event $X_{\vartheta^{(n)}_{i-1}}\in (B,C)$
$\vartheta^{(n)}_{i}-\vartheta^{(n)}_{i-1}\leq\rho^{X_{\vartheta^{(n)}_{i-1}}}_{\bar{\sigma}\sqrt h}$
and so from Lemma \ref{lem4.1},
$\mathbb E[H^m_i]=O(h^m)$.
This together with
the Burkholder--Davis--Gundy inequality and (\ref{convex}) yields
\begin{eqnarray*}\label{4.3}
\mathbb E\left(\max_{1\leq k\leq n}|\vartheta^{(n)}_k-k h|^m\right)&\leq&  O(h^{m/2})+
O(1)\mathbb E\left(\max_{1\leq k\leq n}|\sum_{i=1}^k H_i|^m\right)\nonumber\\
&\leq& O(h^{m/2})+O(1)\mathbb E\left(\left(\sum_{i=1}^n H^2_i\right)^{m/2}\right) \nonumber\\
&\leq& O(h^{m/2})+O(1)n^{m/2-1}\sum_{i=1}^n\mathbb E[|H_i|^m]  =O(h^{m/2})\nonumber
\end{eqnarray*}
and the result follows. \hfill $\square$

\subsection{Some remarks on the proofs}\label{sec:remarks}
\begin{rem}
Theorem \ref{thm2.2} can be easily extended to
American barrier options which we study numerically in Sections 5.4--5.5.

Namely, let $-\infty\leq B<x<C\leq\infty$
and let $Y$ be the unique (in law) solution of (\ref{2.1}).
Let $\tau_{B,C}=T\wedge\inf\{t:Y_t\notin (B,C)\}$. Consider the optimal stopping value
$$V=\sup_{\tau\in\mathcal T_{[0,T]}}\mathbb E[\mathbb{I}_{\tau\leq\tau_{B,C}} f(\tau,X_{\tau})]$$
which corresponds to the price of American barrier options.
Let us notice that if we change in the above formula, the indicator $\mathbb{I}_{\tau\leq\tau_{B,C}}$
to $\mathbb{I}_{\tau<\tau_{B,C}}$ the value remains the same.
As before, we choose $n$ sufficiently large and
a constant $\bar\sigma=\bar\sigma(n)>\sup_{y\in (B,C)}|\sigma(y)|+\sqrt h\sup_{y\in(B,C)}|\mu(y)|$
which satisfies
$\frac{C-x}{\bar\sigma\sqrt h},\frac{x-B}{\bar\sigma\sqrt h}\in\mathbb N\cup\{\infty\}$.
Define
$$\tilde V_n= \max_{\eta\in\mathcal S_n} \mathbb E[\mathbb{I}_{\eta\leq\eta^{(n)}_{B,C}}f(\eta h,X_{\vartheta^{(n)}_{\eta}})]$$
where $\eta^{(n)}_{B,C}=n\wedge\min\{k:Y_{\vartheta^{(n)}_k}\in\{B,C\}\}$.
In this case, we observe that
on the interval $[0,\vartheta^{(n)}_n]$ the process $Y$ can reach the barriers $B,C$ only at the moments
 $\vartheta^{(n)}_i$ $i=0,1,...,n$. Hence, by
similar arguments, Theorem \ref{thm2.2},
can be extended to the current case as well.
\end{rem}
\begin{rem}\label{Skorokhod}
Let us notice that in both of the proofs we get that the random maturity dates $\theta^{(n)}_n, \vartheta^{(n)}_n$
are close
to the real maturity date $T$, and the error estimates are
of order $O(n^{-1/2})$. Hence, even for the simplest case where the diffusion process
$Y$ is the standard Brownian motion we get that
the random variables $|Y_T-Y_{\theta^{(n)}_n}|, |Y_T-Y_{\vartheta^{(n)}_n}|$
are of order $O(n^{-1/4}).$ This is the reason that we can not expect
better error estimates if the proof is done via Skorohod embedding. In \cite{LR}
the authors got error estimates of order $O(n^{-1/2})$ for optimal stopping approximations via Skorokhod embedding, however
they assumed very strong assumptions on the payoff, which excludes even
call and put type of payoffs.

Our numerical results in the next section suggest that the order of convergence we predicted is better than we prove here. In fact, the constant multiplying the power of $n$ is
small, which we think is due to the fact that we have a very efficient way of calculating conditional expectations. We will leave the investigation of this phenomena for future research.

\end{rem}

\section{Numerical examples}\label{sec:5}
In the first subsection we consider two toy models, which are variations of geometric Brownian motion, the first one with capped coefficients and the other one with absorbing
boundaries. In Subsections \ref{sec:CEV}, we consider the CEV model and in \ref{sec:CIR}, we consider the CIR model. In Subsection~\ref{Sec:cbo} we consider the European capped
barrier of \cite{DV} and its American counterpart. We close the paper by considering another modification of geometric Brownian motion with discontinuous coefficients, see
Subsection~\ref{sec:jump}. In these sections we report the values of the option prices, optimal exercise boundaries, the time/accuracy performance of the schemes and numerical
convergence rates.  All computations are implemented in Matlab R2014a on Intel(R) Core(TM) i5-5200U CPU @2.20 GHz, 8.00 GB installed RAM PC.
\subsection{Approximation of one dimensional diffusion with bounded $\mu$ and $\sigma$ by using both two schemes}
We consider the following model:
\begin{equation}\label{5.18}
dS_t=[(A_1 \wedge S_t) \vee B_1]dt+[(A_2 \wedge S_t) \vee B_2]dW_t, \quad  S_0=x,
\end{equation}
where $A_1,B_1,A_2,B_2$ are constants, and the American put option with strike price $\$K$ and expiry date $T$:
\begin{equation}\label{5.19}
v(x)=\sup_{\tau\in\mathcal{T}_{[0,T]}}\mathbb E\left[e^{-r\tau}(K-S_{\tau})^+\right].
\end{equation}
Figure~\ref{fig:exercise-bndry-toy1} shows the value function $v$ and the optimal exercise boundary curve which we obtained using both schemes. Table~\ref{table:tm1} reports
the schemes take and Figure~\ref{fig:spdcntm1} shows the rate of convergence of both models.
\begin{figure}[H]
\begin{center}
\includegraphics[scale=.3]{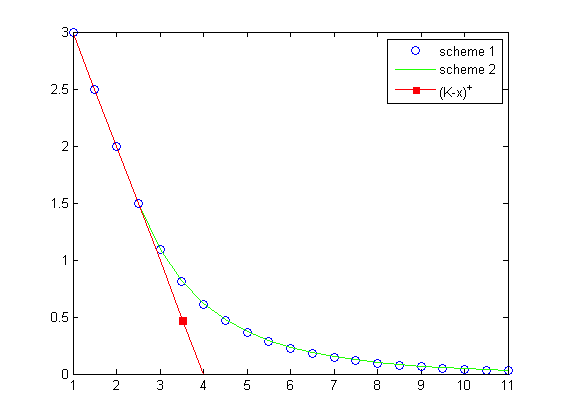}\includegraphics[scale=.3]{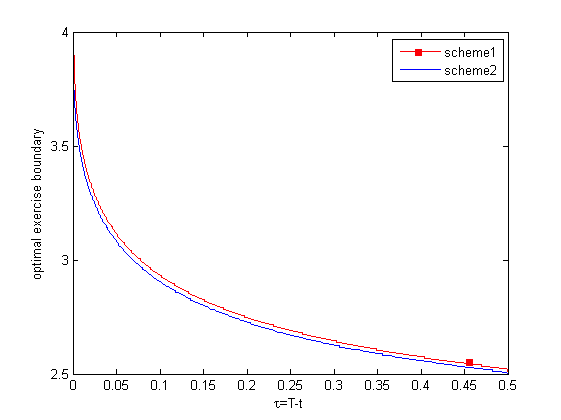}
\end{center}
\caption{The left figure shows that value function of American put (\ref{5.19}) under model (\ref{5.18}) with parameters:   $n=8000,r=0.1,K=4, A_1=10,A_2=10,B_1=2,B_2=2,T=0.5$.
The right figure is the optimal exercise boundary curve. Here $\tau$ is the time to maturity.} \label{fig:exercise-bndry-toy1}
\end{figure}
\vspace{-2em}
\begin{table}[H]
\begin{center}
\caption{Put option prices(\ref{5.19}) under model (\ref{5.18}):}
\begin{tabular}{ c c c c c c } \label{table:tm1}
&scheme 1&scheme 2&&scheme 1&scheme 2\\
\hline
n=1000&0.5954&0.6213&n=2000&0.6042&0.6213\\
\hline
error(\%)&3.64&0.23&error(\%)&2.21&0.08\\
\hline
CPU&0.540s&0.374s&CPU&2.097s&1.467s\\
\hline
&scheme 1&scheme 2&&scheme 1&scheme 2\\
\hline
n=3000&0.6079&0.6214&n=4000&0.6100&0.6216\\
\hline
error(\%)&1.61&0.06&error(\%)&1.27&0.03\\
\hline
CPU&4.749s&3.249s&CPU&8.726s&5.747s\\
\hline
&scheme 1&scheme 2&&scheme 1&scheme 2\\
\hline
n=5000&0.6114&0.6216&n=6000&0.6124&0.6216\\
\hline
error(\%)&1.04&0.03&error(\%)&0.88&0.02\\
\hline
CPU&14.301s&9.145s&CPU&24.805s&14.080s\\
\hline
\end{tabular}
\end{center}
Parameters used in computation are: $r=0.1,K=4, A_1=10,A_2=10,B_1=2,B_2=2,T=0.5,x=4$. Error is computed by taking absolute difference between $v_n(x)$ and $v_{30000}(x)$ and
then dividing by $v_{30000}(x)$.
\end{table}
\begin{figure}[H]
\begin{center}
\includegraphics[scale=.3]{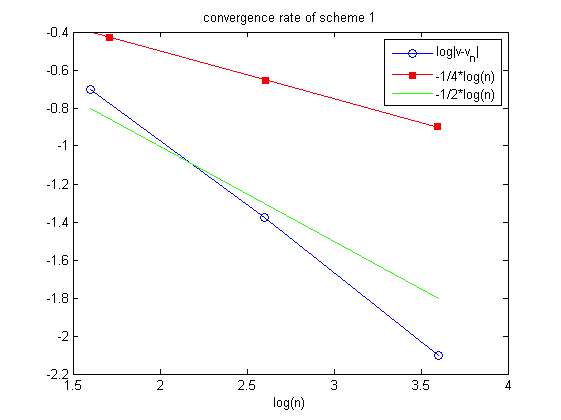}\includegraphics[scale=.3]{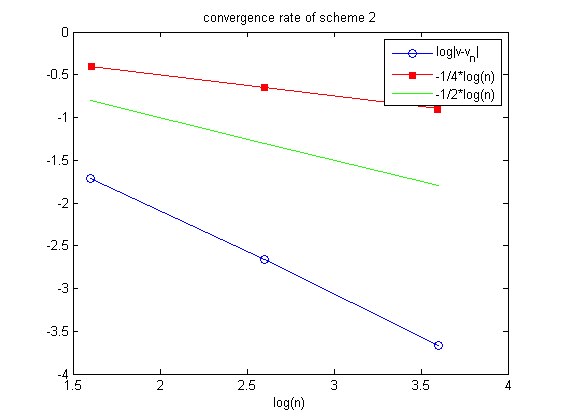}
\end{center}
\caption{Convergence rate figure of both two schemes in the implement of American put (\ref{5.19}) under (\ref{5.18}). We draw $\log{|v-v_n|}\ vs\ \log{n}$ for both schemes to
show the convergence speed.
Here we pick $v(x)=v_{30000}(x)$ for $x=4$ and  $n \in \{40,400, 4000\}$. The slope of the left line given by linear regression is $-0.69981$, and the right is $-0.97422$.
Parameters used in computation are: $r=0.1,K=4, A_1=10,A_2=10,B_1=2,B_2=2,T=0.5,x=4$.}\label{fig:spdcntm1}
\end{figure}

For comparison we will also consider a geometric Brownian motion with absorbing barriers. Let
\begin{equation}
dY_t=Y_tdt+Y_tdW_t, \quad Y_0=x,
\end{equation}
and $B, C\in [-\infty,\infty)$ with $B<x<C$ and denote
\begin{equation}\label{5.21}
X_t=\mathbb{I}_{t<\inf\{s:Y_s\notin (B,C)\}}Y_t+\mathbb{I}_{t\geq\inf\{s:Y_s\notin (B,C)\}}Y_{\inf\{s:Y_s\notin (B,C)\}},  \quad X_0=x.
\end{equation}
As in the previous example, we show below the value function, times the schemes take and the numerical convergence rate.

\begin{figure}[H]
\begin{center}
\includegraphics[scale=.4]{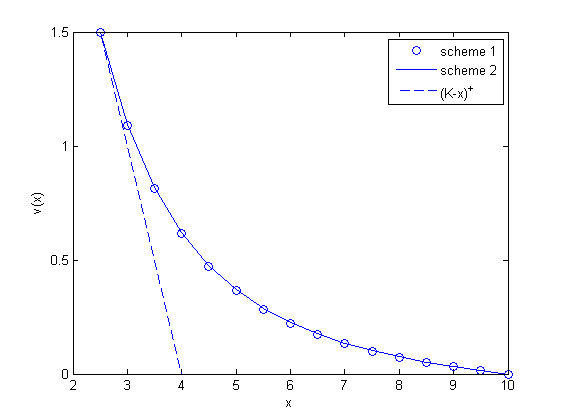}\includegraphics[scale=.4]{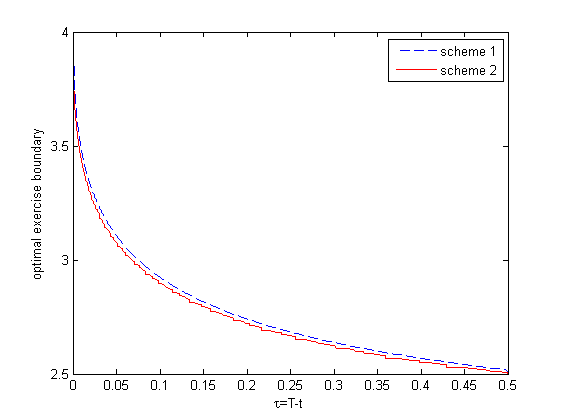}
\end{center}
\caption{The left figure shows the value function of American put (\ref{5.19}) under model (\ref{5.21}) with parameters: $n=30000,r=0.1,K=4, C=10,B=2,T=0.5$. The right figure
is the optimal exercise boundary curve.}\label{fig:toymodel2}
\end{figure}
\begin{figure}[H]
\begin{center}
\includegraphics[scale=.4]{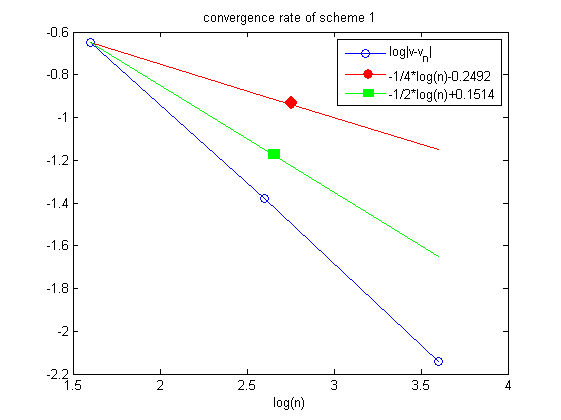}\includegraphics[scale=.4]{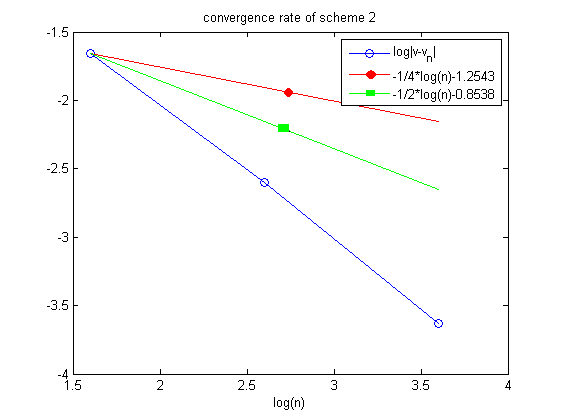}
\end{center}
\caption{Convergence rate figure of both two schemes in the implement of American put (\ref{5.19}) under (\ref{5.21}). We draw $\log{|v-v_n|}\ vs\ \log{n}$ for both schemes to
show the convergence speed. Here we pick, $v(x)=v_{30000}(x)$ for $x=4$ and  $n \in \{40,400, 4000\}$. The slope of the left line given by linear regression is $-0.74513$, the
right one is $-0.98927$.}
\end{figure}

\begin{table}[H]
\caption{put option prices(\ref{5.19}) under (\ref{5.21})}
\begin{center}
\begin{tabular}{ c c c c c c}
&scheme 1&scheme 2&&scheme 1&scheme 2\\
\hline
n=1000&0.5943&0.6175&n=2000&0.6029&0.6185\\
\hline
error(\%)&3.44&0.24&error&2.05&0.09\\
\hline
CPU&0.240s&0.029s&CPU&0.667s&0.091s\\
\hline
&scheme 1&scheme 2&& scheme 1&scheme 2\\
\hline
n=3000&0.6063&0.6184&n=4000& 0.6082&0.6188\\
\hline
error(\%)&1.49&0.09&error&1.17&0.03\\
\hline
CPU&1.220s&0.157s&CPU&1.927s&0.254s\\
\hline
&scheme 1&scheme 2&&scheme 1&scheme 2\\
\hline
n=5000&0.6095&0.6188&n=6000&0.6104&0.6189\\
\hline
error(\%)&0.96&0.04&error&0.81&0.02\\
\hline
CPU&2.720s&0.359s&CPU&3.723s&0.457s\\
\hline
\end{tabular}
\end{center}
\textbf{Note:} Parameters used in computation: $r=0.1,K=4, C=10,B=2,T=0.5, x=4$. Error is computed by taking absolute difference between $v_n(x)$ and $v_{30000}(x)$ and then
dividing by $v_{30000}(x)$.\\
\end{table}

\subsection{Pricing American put option under CEV}\label{sec:CEV}
Set $\beta \leq 0$ and consider the CEV model:
\begin{equation}\label{5.1}
dS_t=rS_t+\delta S_t^{\beta+1}dW_t, \quad S_0=x,
\end{equation}
and let $\delta=\sigma_0x^{-\beta}$ (to be consistent with the notation of \cite{DV}). Consider the American put option pricing problem:
\begin{equation}\label{5.2}
v_A(x)=\sup_{\tau\in\mathcal{T}_{[0,T]}} \mathbb{E}\left[e^{-r \tau}(K-S_{\tau})^+\right].
\end{equation}

Clearly, the CEV model given by (\ref{5.1}) does not satisfies our assumptions.
However, this limitation can be solved by truncating the model.
The next result shows that if we choose absorbing barriers $B,C>0$ such that $B$ is small enough and
$C$ is large enough, then the optimal stopping problem does not change much.
Hence, we can apply our numerical scheme for the absorbed diffusion
and still get error estimates of order $O(n^{-1/4})$.
\begin{lem}
Choose $0<B<C$.
Consider the CEV model given by (\ref{5.1}) and let $X$ be the absorbed process
$$X_t=\mathbb{I}_{t<\inf\{s:S_s\notin (B,C)\}}S_t+\mathbb{I}_{t\geq\inf\{s:S_s\notin (B,C)\}}S_{\inf\{s:S_s\notin (B,C)\}}, \ \ t\geq 0.$$
Then
\begin{equation}
\sup_{\tau\in\mathcal T_{[0,T]}}\mathbb E[e^{-r\tau}(K-S_{\tau})^{+}]-\sup_{\tau\in\mathcal T_{[0,T]}}\mathbb E[e^{-r\tau}(K-X_{\tau})^{+}]=  \mathbb{I}_{\tau^* \leq \tau_{B}}
B+O(C^{-k}),
\ \ \forall k\geq 1,
\end{equation}
where $\tau^*$ is the optimal time for first the expression on the left and $\tau_B$ is the first time $S_t$ is below the level $B$.\footnote{It might seem to be a moot point
to have $\mathbb{I}_{\tau^* \leq \tau_{B}}$ since this involves determining the optimal stopping boundary. But one can actually easily determine $B$'s satisfying this condition
through through setting $B \leq b$, the perpetual version of the problem whose stopping boundary, which can be determined by finding the unique root of
\begin{equation}\label{eq:perpbnd}
F(x)=\phi'(x)\frac{K-x}{\phi(x)}+1,
\end{equation}
 in which the function $\phi$ is given by \cite[Equation (38)]{DV} (with $\lambda=r$), see e.g. \cite{doi:10.1080/14697680903170817}.
The perpetual Black-Scholes boundary,
\begin{equation}\label{eq:perbndbs}
b=2rK/(2r+\delta^2)
\end{equation}
 is in fact is a lower bound to $b$ for $\beta<0$, which can be proved using the comparison arguments in \cite{MR3180040}.}
\end{lem}

\begin{proof}
Let us briefly argue that
$\mathbb E\left(\max_{0\leq t\leq T}S^k_t\right)<\infty$ for all $k$. Clearly, if $\beta=0$ then $S$ is a geometric Brownian motion and so the
statement is clear. For $\beta<0$ we observe that the process $Y=S^{-\beta}$ belongs
to CIR family of diffusions (see equation (1) in \cite{GJY}). Such process can be represented in terms
of BESQ (squared Bessel processes); see equation (4) in \cite{GJY}. Finally, the moments of the running maximum of the (absolute value) Bessel process
can be estimated by Theorem 4.1 in \cite{GP}.
Hence from the Markov inequality it follows that
$\mathbb P(\max_{0\leq t\leq T}S_t\geq M)=O(M^{-k}).$
 Next, consider the stopping time $\tau_{B}=\inf\{t: S_t=B\}\wedge T$.
If $\tau_{B}<T$ then the payoff of the put option is $K-B$,
and since $S$ is non--negative its $B$ optimal to stop at this time.
It is clear also that one would make no error if $B$ is always in the stopping region.
\end{proof}
\begin{table}[H]
\begin{center}
\caption{Pricing American puts (\ref{5.2}) under the CEV (\ref{5.1}) model with different $\beta$'s by using the first scheme}
\begin{tabular}{ c c c c c}
 $K$ & $\beta=-1$  &$\beta=-\frac{1}{3}$ \\
90&1.4373(1.5122)&1.3040(1.3844)\\
100&4.5359(4.6390)&4.5244(4.6491)\\
110& 10.6502(10.7515)&10.7716(10.8942)\\
\end{tabular}
\end{center}
Parameters used in computation are: $\sigma_0=0.2, x=100, T=0.5, r=0.05, n=15000,B=0.01,C=200$. The values in the parentheses are the results from last row of Tables $4.1, 4.2,
4.3$ from paper \cite{WZ}, which carries out the artificial boundary method discussed in the introduction. This is a finite difference method which relies on artificially introducing an exact boundary condition. 
\end{table}
\begin{table}[H]
\begin{center}
\caption{Prices of the American puts (\ref{5.2}) under (\ref{5.1}) with different $\beta$'s by using the second scheme}
\begin{tabular}{ c c c c }
 $K$ & $\beta=-1$  &$\beta=-\frac{1}{3}$\\
90&1.5123(1.5122)&1.3845(1.3844)\\
100&4.6392(4.6390)&4.6492(4.6489)\\
110&10.7517(10.7515)&10.8943(10.8942)\\
\end{tabular}
\end{center}
Parameters used in computation are: $\sigma_0=0.2, x=100, T=0.5, r=0.05, n=15000, B=0.01, C=200$. The values in the parentheses are the results from last row of Tables $4.1,
4.2, 4.3$ from paper \cite{WZ}, which carries out the artificial boundary method discussed in the introduction. This is a finite difference method which relies on artificially introducing an exact boundary condition. 
\end{table}
\begin{figure}[H]
\begin{center}
\includegraphics[scale=.4]{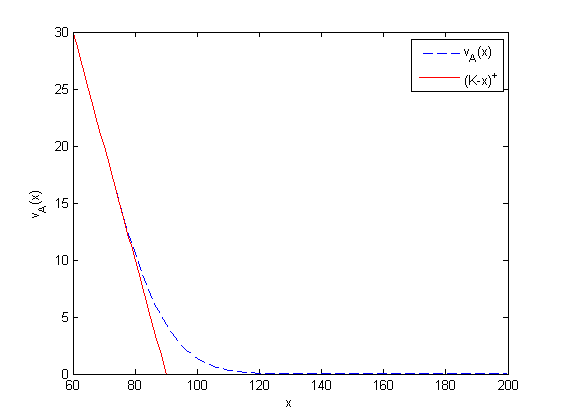}\includegraphics[scale=.4]{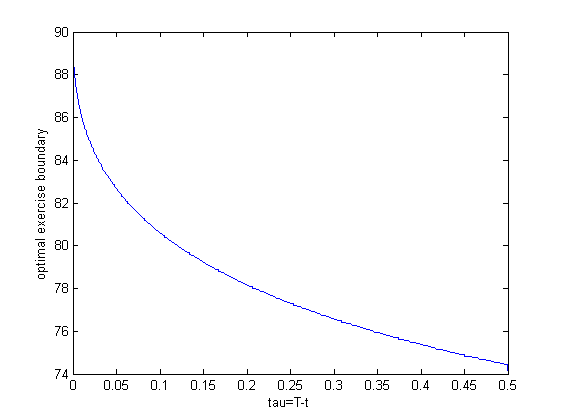}
\end{center}
\caption{ The left figure shows that value function of American put (\ref{5.2}) under CEV (\ref{5.1}) with parameters:
$n=2000,\beta=-\frac{1}{3},\delta=\sigma_0(100)^{-\beta},\sigma_0= 0.2, T = 0.5,r = 0.05,B=0.01,C = 200,K = 90.$ The right figure is the optimal exercise boundary curve which
$\tau=T-t$ denotes time to maturity.  With these parameters $F(60)=-0.3243, F(70)=0.1790$ (where $F$ is defined in \eqref{eq:perpbnd}), from which we conclude that the unique
root is in (60,70). So there in fact is no error in introducing a lower barrier.  In this case $b=9.3577$ (where $b$ is defined in \eqref{eq:perbndbs}).}
\end{figure}
Since we observed that second scheme exhibits faster convergence, we report the CPU time and errors for second scheme.
\begin{table}[H]
\begin{center}
\caption{Prices of the American puts under (\ref{5.1}) with different $beta$'s (second scheme)}
\begin{tabular}{ c c c c }
 & $\beta=-1$  &$\beta=-\frac{1}{3}$\\
K=90&1.5147&1.3862\\
CPU&0.335s&0.259s\\
error(\%)&0.16&0.12\\
\hline
K=100&4.6414&4.6383\\
CPU&0.314s&0.256s\\
error(\%)&0.05&0.23\\
\hline
K=110&10.7523&10.8925\\
CPU&0.338s&0.253s\\
error(\%)&0.005&0.02\\
\end{tabular}
\end{center}
Parameters used in computation are: $\sigma_0=0.2, x=100, T=0.5, r=0.05, B=0.01, C=200$. The Error is computed by taking absolute difference between $v_n(x)$ and $v_{15000}(x)$
and then dividing by $v_{15000}(x)$.
\end{table}
\
\begin{table}[H]\label{tab:comp}
\begin{center}
\caption{Pricing American puts (\ref{5.2}) under (\ref{5.1}) with different $\sigma_0$ (second scheme)}
\begin{tabular}{ c c c c c}
 & $\sigma_0=0.2$  &$\sigma_0=0.3$&$\sigma_0=0.4$ \\
K=35&1.8606(1.8595)&4.0424(4.0404)&6.4017(6.3973)\\
CPU&0.247s&0.171s&0.136s\\
error(\%)&0.04&0.03&0.06\\
\hline
K=40&3.3965(3.3965)&5.7920(5.7915)&8.2584(8.2574)\\
CPU&0.237s&0.172s&0.154s\\
error(\%)&0.01&0.002&0.004\\
\hline
K=45& 5.9205(5.9204)&8.1145(8.1129)&10.5206(10.5167)\\
CPU&0.203s&0.170s&0.146s\\
error(\%)&0.01&0.02&0.03\\
\end{tabular}
\end{center}
Parameters used in computation are: $T=3, r=0.05,x=40,\beta=-1,n=100, B=0.01, C=100$. Error is computed by taking absolute difference between $v_n(x)$ and $v_{30000}(x)$ and then dividing by $v_{30000}(x)$. The values in the parentheses are the results computed by FDM:$1024\times 1024$ reported in Table 1 of \cite{WZ2}.
\end{table}

In Table~\ref{tab:comp} we compare our numerical results with the ones reported in Table 1 of \cite{WZ2}, the numerical values here are closer to results computed by FDM:$1024\times 1024$, which is the traditional Crank-Nicolson finite difference scheme (the $1024 \times 1024$ refer to the size of the lattice). The CPU time of FDM:$1024\times 1024$ is 6.8684s. Table 1 of \cite{WZ2} also reports the performance of the Laplace-Carson transform and the artificial boundary method of \cite{WZ} (both of which we briefly described in the introduction). The CPU for Laplace-Carson method is 0.609s, ABC:$512\times 512$  is 0.9225s. If we consider FDM:$1024\times 1024$ as a Benchmark, our method is more than \textbf{30 times} faster with accuracy up to 4 decimal points. Comparing with ABC:$512\times 512$ method, our method is about \textbf{5 times} faster with the same accuracy.  Our method is about \textbf{3 times} faster than what the Laplace transform method produced when it had accuracy up to 3 decimal points. (Here accuracy represents the absolute difference between the computed numerical values and the FDM:$1024\times 1024$  divided by the FDM:$1024\times 1024$).
\subsection{Pricing American put option under CIR}\label{sec:CIR}
In this subsection, we use numerical scheme to evaluate American put options $v_A(x)$ (recall formula (\ref{5.2})) under CIR model. (This model is used for pricing VIX options,
see e.g. \cite{2016arXiv160600530K}).
Let the volatility follow
\begin{equation}\label {5.3b*}
dY_t=(\beta- \alpha Y_t)dt+\sigma \sqrt{Y_t}dW_t, \quad Y_0=y,
\end{equation}
Consider the absorbed process
$$X_t=\mathbb{I}_{t<\inf\{s:Y_s\notin (B,C)\}}Y_t+\mathbb{I}_{t\geq\inf\{s:Y_s\notin (B,C)\}}Y_{\inf\{s:Y_s\notin (B,C)\}}, \quad t\geq 0,$$
for $0<B<C$.
Same arguments as in Section~\ref{sec:CEV} give that for CIR process
$\mathbb P(\max_{0\leq t\leq T}Y_t\geq C)=O(C^{-k})$, $k\geq 1$. Moreover, from Theorem 2 in \cite{GJY}
it follows that
$\mathbb P(\min_{0\leq t\leq T}Y_t\leq B)=O(B^{2\nu})$ for $\nu=\frac{2 \beta}{\sigma^2}-1>0$. Thus if we consider the absorbed diffusion $X$ then
the change of the value of the option price is bounded by $O(C^{-k})+O(B^{2\nu})$ for all $k$. (In fact, potentially we do not make any error by having an absorbing boundary if
$B$ is small enough, as we argued in the previous section.)
\begin{table}[H]
\begin{center}
\caption{Prices of American puts (\ref{5.2})  under (\ref{5.3b*}) with different $K$}
\begin{tabular}{ c c c c }
 & $\mbox{scheme 1}$  &$\mbox{scheme 2}$\\
$K=35$&4.4654&4.5223\\
\hline
$K=40$& 8.1285& 8.1932\\
\hline
$K=45$&12.4498&12.5167\\
\end{tabular}
\end{center}
Parameters used in computation are: $\sigma=2, y=40, T=0.5, r=0.1, B=0.01, C=200,\beta=2,\alpha=0.5,n=30000$.
\end{table}

As the previous example, since the second scheme looks more efficient, we only report its time with the accuracy performance here.
\begin{table}[H]
\begin{center}
\caption{American put prices  under (\ref{5.3b*}) for a variety of $K$'s using the second scheme}
\begin{tabular}{ c c c c c}
 & $K=35$  &$K=40$&$K=45$ \\
$n=100$&4.5140&8.1834&12.5127\\
CPU&0.245s&0.202s&0.198s\\
error(\%)&0.18&0.12&0.03\\
\hline
$n=500$&4.5215&8.1918&12.5163\\
CPU&0.600s&0.585s&0.579s\\
error(\%)&0.02&0.02&0.003\\
\hline
$n=1000$&4.5238&8.1925&12.5170\\
CPU&0.807s&0.813s&0.839s\\
error(\%)&0.03&0.01&0.002\\
\end{tabular}
\end{center}
Parameters used in computation are: $\sigma=2, y=40, T=0.5, r=0.1, B=0.01, C=200,\beta=2,\alpha=0.5$. Error is computed by taking absolute difference between $v_n(x)$ and
$v_{30000}(x)$ and then dividing by $v_{30000}(x)$.
\end{table}

\subsection{Pricing double capped barrier options under CEV by using second scheme}\label{Sec:cbo}
In this subsection, we use second scheme to evaluate double capped barrier call options under the CEV model given by (\ref{5.1}).
Let us denote the European option value by
\begin{equation}\label{5.10}
v_{E}(x)=e^{-rT}\mathbb{E}\left[\mathbb{I}_{T<\tau^*}(S_T-K)^+\right],
\end{equation}
and the American option value by
\begin{equation}\label{5.11}
v_{A}(x)=\sup_{\tau\in\mathcal{T}_{[0,T]}} \mathbb{E}\left[\mathbb{I}_{\tau<\tau^*}e^{-r\tau}(S_{\tau}-K)^+\right],
\end{equation} where $\tau^*=\inf\{t\geq 0; S_t\ne(L, U)\}$.
We use the parameters values given in \cite{DV} and report our results in Table~\ref{table:double-capped} and Figure~\ref{fig:dceev}.
\begin{table}[H]\label{table:double-capped}
\caption{Double capped European Barrier Options (\ref{5.10}) under CEV(\ref{5.1}) (second scheme)}
\begin{center}
\begin{tabular}{ c c c c c c c}
 $U$ & $L$ & $K$ & $\beta=-0.5$  &$\beta=0$&$\beta=-2$&$\beta=-3$\\
\hline
 120 & 90 & 95& 1.9012&1.7197&2.5970&3.2717\\
\hline
&CPU&&0.970s&1.026s&6.203s&1.209s\\
\hline
&error&&0.96\%&1.2\%&0.4\%&2.7\%\\
\hline
120 & 90 & 100& 1.1090&0.9802&1.6101&2.0958\\
\hline
&CPU&&0.957s&1.102s&6.066s&1.205s\\
\hline
&error&&0.98\%&1.2\%&0.44\%&2.3\%\\
\hline
120 & 90 & 105& 0.5201& 0.4473&0.8142&1.1072\\
\hline
&CPU&&0.961s&1.023s&6.205s&1.187s\\
\hline
&error&&1.00\%&1.3\%&0.53\%&2.2\%\\
\hline
&n&&2000&2000&5000&2000\\
\hline
\end{tabular}
\end{center}
\caption{Double capped American Barrier Options (\ref{5.11}) under CEV(\ref{5.1}) (second scheme)}
\begin{center}
\begin{tabular}{ c c c c c c c}
 $U$ & $L$ & $K$ & $\beta=-0.5$  &$\beta=0$&$\beta=-2$&$\beta=-3$\\
\hline
 120 & 90 & 95& 9.8470&9.8271&9.9826&10.0586\\
\hline
&CPU&&0.965s&1.007s&1.248s&1.170s\\
\hline
&error&&0.55\%&0.63\%&0.56\%&0.81\%\\
\hline
120 & 90 & 100& 7.4546&7.4522&7.5118&7.5203\\
\hline
&CPU&&0.942s&1.017s&1.239s&1.211s\\
\hline
&error&&0.70\%&0.8\%&0.47\%&0.57\%\\
\hline
120 & 90 & 105& 5.2612& 5.2788&5.2345&5.1669\\
\hline
&CPU&&0.968s&1.044s&1.145s&1.188s\\
\hline
&error&&1.00\%&1.2\%&0.35\%&0.17\%\\
\hline
&n&&2000&2000&2000&2000\\
\hline
\end{tabular}
\end{center}
Parameters used in calculations are: $x=100, \sigma(100)=0.25$ (denoting $\sigma(S)=\delta S^{1+\beta}$), $ r=0.1, T=0.5$. Error is computed by taking the absolute difference
between $v_n$ and $v_{40000}$ and then dividing by $v_{40000}$.
\end{table}

\begin{figure}[H]\label{fig:dceev}
\begin{center}
\includegraphics[scale=.4]{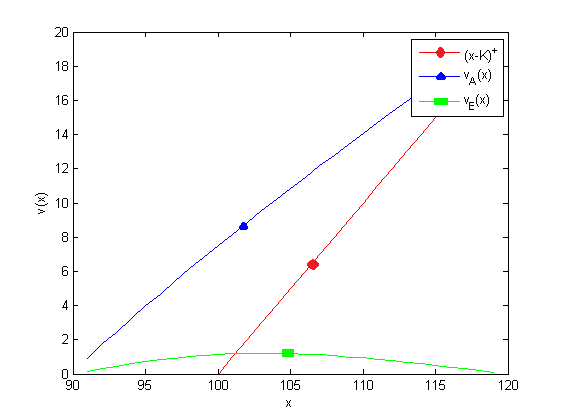}\includegraphics[scale=.4]{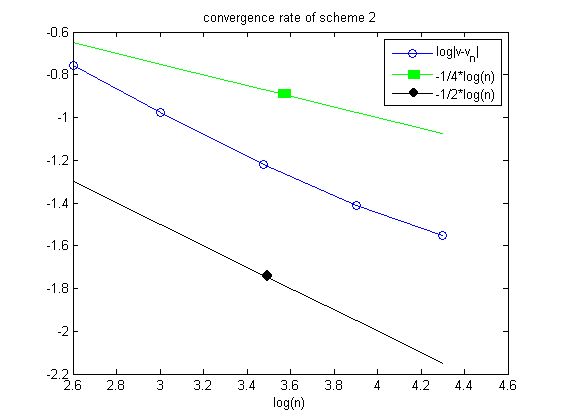}
\end{center}
\caption{The left figure shows the value function $v_A(x)$(\ref{5.11}) and $v_E(x)$(\ref{5.10}) with Parameters: $\delta=2.5, r=0.1, T=0.5, \beta=-0.5, n=5000, L=90, U=120,
K=100$,. In the right figure, we show the $\log{|v-v_n|}$ vs $\log{n}$ picking $v(x)=v_{40000}(x)$, $x=100$, and $n\in \{400, 1000, 3000, 8000, 20000\}$. The slope of the blue
line given by linear regression is $-0.47178$.}
\end{figure}

\subsection{Pricing double capped American barrier options with jump volatility by using second scheme}\label{sec:jump}
In this subsection, we show the numerical results for pricing double capped American barrier option when the volatility has a jump.
Consider the following modification of a geometric Brownian motion:
\begin{equation}\label{5.16}
dS_t=rS_tdt+\sigma(S_t)S_tdW_t, S_0=x
\end{equation}
where $\sigma(S_t)=\sigma_1$, for $S_t\in[0,S_1]$, $\sigma(S_t)=\sigma_2$, for $S_t\in[S_1,\infty)$.
We will compute the American option price
\begin{equation}\label{5.17}
v(x)=\sup_{\tau\in\mathcal{T}_{[0,T]}}\mathbb{E}\left[e^{-r\tau}\mathbb I_{\tau<\tau^*}(S_{\tau}-K)^+\right],
\end{equation}
where $\tau^*=\inf\{t\geq 0; S_t\ne(L, U)\}$. We use scheme 2 here, since it can handle discontinuous coefficients; see Figures~\ref{fig:disc} and \ref{fig:conv-rate}. We
compare it to the geometric Brownian motion (gBm) problem with different volatility parameters. In particular, we observe that although the jump model's price and the gBm model
with average volatility are close in price, the optimal exercise boundaries differ significantly.

\begin{figure}[H]\label{fig:disc}
\begin{center}
\includegraphics[scale=.4]{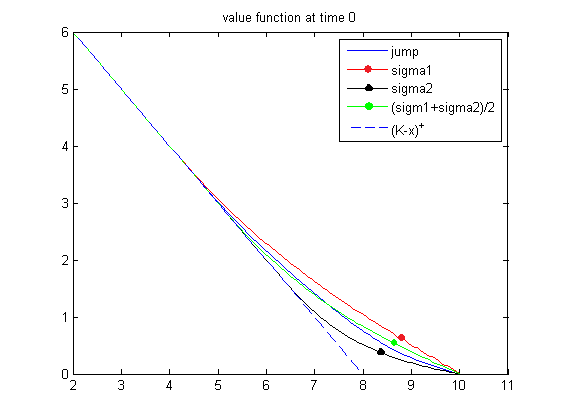}\includegraphics[scale=.4]{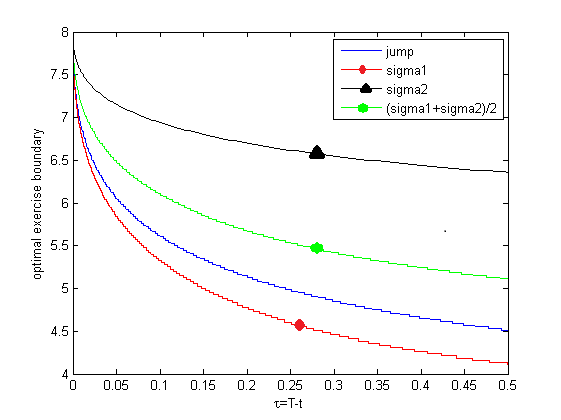}
\end{center}
\caption{There are four curves in the left figure. They are value functions (\ref{5.17}) under model (\ref{5.16}) with different values of $\sigma(S_t)$ as specified in the
legend. The parameters used in computations are: $n=4000, r=0.1, S_1= K=8, L=2, U=10, \sigma_1=0.7, \sigma_2=0.3, T=0.5$. Here $\sigma_2<\frac{\sigma_1+\sigma_2}{2}<\sigma_1$,
and as expected the gBm option price decreases as $\sigma$ decreases but the gBm price with jump volatility and one corresponding to gBm with $\sigma=\sigma_2$ intersect
in the interval $[7,8]$. The right figure is of the optimal exercise boundary curves.}
\end{figure}

\begin{figure}[H]\label{fig:conv-rate}
\begin{center}
\includegraphics[scale=.4]{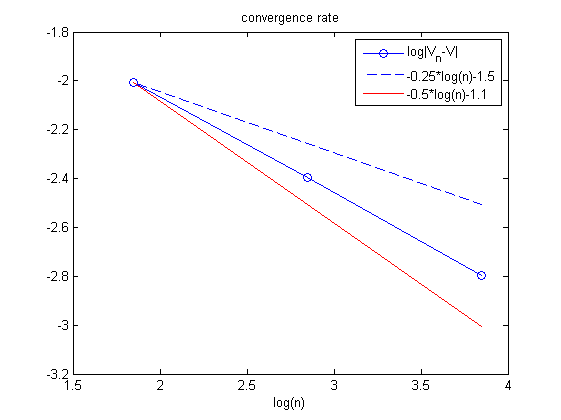}
\end{center}
\caption{Convergence rate figure of scheme 2 in the implement of double capped American put (\ref{5.17}) under (\ref{5.16}). We pick $v(x)=v_{40000}(x)$ to be the actual price,
and the other three values are $v_{70}(x),v_{700}(x), v_{7000}(x), x=8$. The slope of the blue line given by linear regression approach is $-0.39499$. Parameters used in
computation: $n=4000, r=0.1, K=8, L=2, U=10, \sigma_1=0.7, \sigma_2=0.3, S1=8, T=0.5$.}
\end{figure}

\newpage
\bibliography{ref}{}
\bibliographystyle{siam}

\end{document}